\documentclass[12pt,a4paper]{article}
\usepackage{amsmath}
\usepackage{amsfonts}
\usepackage{amssymb}
\usepackage{amsthm}
\usepackage[table]{xcolor}
\usepackage[colorlinks=true,linkcolor=blue,citecolor=blue]{hyperref}
\usepackage{enumerate}
\usepackage{graphicx}
\usepackage{sudoku}
\usepackage{tikz}
\usetikzlibrary{arrows}
\usepackage{caption}
\usepackage{subcaption}
\usepackage{geometry,rotating}
\usepackage[boxed]{algorithm2e}
\usepackage{framed}
\usepackage{multirow}
\usepackage{chngpage} % allows for temporary adjustment of side margins

\newtheorem{theorem}{Theorem}[section]
\newtheorem{proposition}{Proposition}[section]

\theoremstyle{remark}
\newtheorem{remark}{Remark}[section]
\newtheorem{example}{Example}[section]

\newcommand\argmin{\mathop{\rm argmin}}
\newcommand{\diag}{\operatorname{diag}}
\newcommand{\rank}{\operatorname{rank}}

\newcommand{\qede}{\hspace*{\fill}$\Diamond$\medskip}

\newcommand{\R}{\mathbb{R}}

\newcommand{\eps}{\varepsilon}
\newcommand{\tr}{\operatorname{tr}}

\title{Douglas--Rachford Feasibility Methods for\\ Matrix Completion Problems\footnote{All authors are at CARMA, University of Newcastle, Callaghan, NSW 2308, Australia.}}
\author{Francisco J. Arag\'on Artacho\thanks{Email: \url{francisco.aragon@ua.es}}
  \and Jonathan M. Borwein\thanks{Email: \url{jon.borwein@gmail.com}}
  \and Matthew K. Tam\thanks{Email: \url{matthew.k.tam@gmail.com}}}

\begin{document}

\maketitle

\begin{abstract}
 In this paper we give general recommendations for successful application of the Douglas--Rachford reflection method  to convex and non-convex real matrix-completion problems. These guidelines are demonstrated by various illustrative examples.
\end{abstract}



\section{Introduction}
 \emph{Matrix completion} may be posed as an inverse problem in which a matrix possessing certain  properties is to be reconstructed knowing only a subset of its entries. A great  many problems can be usefully cast within this framework (see \cite{J90,L01} and the references therein).

By encoding each of the properties which the matrix possesses along with its known entries as constraint sets, matrix completion can be cast as a \emph{feasibility problem}. That is, it is reduced to the problem of finding a point contained in the intersection of a (finite) family of sets.

\emph{Projection algorithms} comprise a class of general purpose
iterative methods which are frequently used to solve feasibility
problems (see \cite{BB96} and the references therein). At each step,
these methods utilize the nearest point projection onto each of the
individual constraint sets. The philosophy here is that it is
simpler to consider each constraint separately (through their
nearest point projections), rather than the intersection directly.

Applied to closed convex sets in Euclidean space, the behaviour of
projection algorithms is quite well understood. Moreover, their
simplicity and ease of implementation has ensured continued
popularity for successful applications in a variety of non-convex
optimization and reconstruction problems \cite{BCL02,BCL03,ABT13}.
This is despite the absence of sufficient theoretical justification,
although there are some useful beginnings \cite{BS11,AB12,HL12}. In
many of these settings the \emph{Douglas--Rachford} method (see
Section~\ref{sec:DR}) has been observed to perform better than other
projection algorithms, and hence will be the  projection algorithm
of choice for this paper. A striking example of its better behaviour
is detailed in Section \ref{sec:prot}.

We do note that there are many other useful projection algorithms, and many applicable variants. See for example, the \emph{method of cyclic projections}~\cite{BB93,BBL97}, \emph{Dykstra's method} \cite{BD86,BB94,BR05}, the \emph{cyclic Douglas--Rachford method} \cite{BT13}, and many references contained in these papers.

  \begin{quote}In a recent paper \cite{ABT13}, the present authors observed that many successful non-convex applications of the Douglas--Rachford method can be considered as matrix completion problems. The aim of this paper is to give general guidelines for successful application of the Douglas--Rachford method to a variety of (real) matrix reconstruction problems, both convex and non-convex.\end{quote}

\bigskip

 The remainder of the paper is organised as follows: In Section \ref{sec:pre} we first recall what is known about the Douglas--Rachford method, and then discuss our modelling philosophy.
In Section~\ref{sec:conv} we consider several matrix completion problems in which all the constraint sets are convex:
positive semi-definite matrices, doubly-stochastic matrices, Euclidean distance matrices; before discussing adjunction of noise.
 This is followed in Section~\ref{sec:nonconv} by a more detailed description of several classes in which  some of the constraint sets are non-convex. In the first two subsections, we first look at low-rank constraints and then at low-rank Euclidean distance problems. In  Section~\ref{sec:prot} we present a first detailed application by  viewing protein reconstruction from NMR data as a low-rank Euclidean distance problem. The final three subsections  of  Section~\ref{sec:nonconv} carefully consider Hadamard, skew-Hadamard and circulant-Hadamard matrix problems, respectively.  We end with various concluding remarks in Section~\ref{sec:conc}.

\section{Preliminaries}\label{sec:pre}
Let $E$ denote a finite dimensional Hilbert space -- a \emph{Euclidean space}. We will typically be considering the Hilbert space of real $m\times n$ matrices whose inner product is given by
 $$\langle A,B\rangle :=\tr(A^TB),$$
where the superscript $T$ denotes the transpose, and $\tr(\cdot)$ the trace of an $n\times n$ square matrix. The induced norm is the \emph{Frobenius norm} and can be expressed as
 $$\|A\|_F:=\sqrt{\tr(A^TA)}=\sqrt{\sum_{i=1}^n\sum_{j=1}^ma_{ij}^2}.$$

A \emph{partial (real) matrix} is an $m\times n$ array for which only entries in certain locations are known. A \emph{completion} of the partial matrix $A=(a_{ij})\in\R^{m\times n}$, is a matrix $B=(b_{ij})\in\R^{m\times n}$ such that if $a_{ij}$ is specified then $b_{ij}=a_{ij}$. The problem of \emph{(real) matrix completion} is the following: \emph{Given a partial matrix find a completion having certain properties of interest (e.g. positive semi-definite).}

Throughout this paper, we formulate matrix completion as a \emph{feasibility problem}. That is,
 \begin{equation} \label{eq:feasibility}
  \text{Find }X\in \bigcap_{i=1}^NC_i\subseteq\R^{m\times n}.
 \end{equation}
Let $A$ be the partial matrix to be completed. We will take $C_1$ to be the set of all completions of $A$, and the sets $C_2,\dots,C_N$ will be chosen such that their intersection has the properties of interest. In this case, (\ref{eq:feasibility}) is precisely the problem of matrix completion.

\subsection{The Douglas--Rachford method}\label{sec:DR}

Recall that the \emph{nearest point projection onto $S\subseteq E$} is the (set-valued) mapping $P_S:E\to 2^S$ defined by
 $$P_Sx:=\argmin_{s\in S}\|s-x\|=\{s\in S:\|s-x\|=\inf_{y\in S}\|y-x\|\}.$$
The \emph{reflection with respect to $S$} is the (set-valued) mapping $R_S:E\to 2^E$ defined by
 $$R_S:=2P_S-I,$$
where $I$ denotes the identity map.

In an abuse of notation, if $P_Sx$ (resp. $R_Sx$) is singleton we use $P_Sx$ (resp. $R_Sx$) to denote the  unique nearest point.

We now recall what is know about the Douglas--Rachford method, as specialized to finite dimensional spaces.

\begin{theorem}[Convex Douglas--Rachford iterations]\label{th:DR}
 Suppose $A,B\subseteq E$ are closed and convex. For any $x_0\in E$ define
  $$x_{n+1}:=T_{A,B}x_n\text{ where }T_{A,B}:=\frac{I+R_BR_A}{2}.$$
 Then if:
  \begin{enumerate}[(a)]
   \item $A\cap B\neq\emptyset$, $(x_n)$ converges to a point $x$ such that $P_Ax\in A\cap B$.
   \item $A\cap B=\emptyset$, $\|x_n\|\to+\infty$.
  \end{enumerate}
\end{theorem}
\begin{proof}
 See, for example, \cite[Th.~3.13]{BCL04}.
\end{proof}

The results of Theorem~\ref{th:DR} can only be directly applied to the problem of finding a point in the intersection of two sets. For matrix completion problems formulated as feasibility problems with greater than two sets, we use a well known \emph{product space reformulation.}

\begin{example}[Product space reformulation]\label{ex:product}
For constraint sets $C_1,C_2,\dots,C_N$ define\footnote{The set $D$ is sometimes called the \emph{diagonal}.}
 $$D:=\{(x,x,\dots,x)\in E^N|x\in E\},\quad C:=\prod_{i=1}^NC_i.$$
We now have an equivalent feasibility problem since
 $$ x\in\bigcap_{i=1}^NC_i \iff (x,x,\dots,x)\in D\cap C. $$
Moreover, $T_{D,C}$ can be readily computed whenever $P_{C_1},P_{C_2},\dots,P_{C_N}$ can be since
 $$P_Dx=\left(\frac{1}{N}\sum_{i=1}^Nx_i\right)^N,\quad P_Cx=\prod_{i=1}^N P_{C_i}x_i.$$
For further details see, for example, \cite[Section~3]{ABT13}. \qede
\end{example}

In the non-convex setting there are some useful theoretical beginnings. For a Euclidean sphere and affine subspace, with the reflection performed first with respect the sphere, Borwein and Sims \cite{BS11} show that, appropriately viewed, the Douglas--Rachford scheme converges locally. An explicit region of convergence was given by Arag\'on Artacho and Borwein \cite{AB12} for $\R^2$. Hesse and Luke \cite{HL12} obtained local convergence results using a relaxed local version of firm nonexpansivity and appropriate regularity conditions, assuming that the reflection is performed first with respect to a subspace. We note that varying the order of the reflection does not  make a substantive difference.

\subsection{Modelling philosophy}
As illustrated for Sudoku and other NP-complete combinatorial problems in \cite{ABT13}, there are typically many ways to formulate the constraint set for a given matrix completion problem. For example, by choosing different sets $C_2,C_3,\dots,C_N$, in (\ref{eq:feasibility}), such that $\cap_{i=2}^NC_i$ has the properties of interest. To apply the Douglas--Rachford method, these sets will be chosen in such a way that their individual nearest point projections are succinctly simple to compute --- ideally in closed form. There is frequently a trade-off between the number of sets in the intersection, and the simplicity of their projections. For example, one extreme would be to encode the property of interest in a single constraint set. In this case, it is likely that its projection is difficult to compute.

To illustrate this philosophy, consider the following example which we revisit in Section~\ref{sec:doublystoc}. Suppose the property of interest is the constraint
 $$\left\{X\in\R^{m\times n}|X_{ij}\geq 0,\sum_{k=1}^nX_{kj}=1\text{ for }i=1,\dots,m\text{ and }j=1,\dots,n\right\}.$$
This set is equal to the intersection of $C_2$ and $C_3$ where
 \begin{align*}
   C_2 &:= \left\{X\in\R^{m\times n}|X_{ij}\geq 0\text{ for }i=1,\ldots,m\text{ and }j=1,\ldots,n\right\},\\
   C_3 &:= \left\{X\in\R^{m\times n}|\sum_{i=1}^nX_{ij}=1\text{ for }j=1,\ldots,n\right\}.
 \end{align*}
Here the projections onto the cone $C_2$ and the affine space $C_3$ can be easily computed (see Section~\ref{sec:doublystoc}). In contrast, the projection directly onto $C_2\cap C_3$ amounts to finding the nearest point in the convex hull of the set of matrices having a one in each row and remaining entries zero. This projection is less straightforward, and has no explicit form. For details, see \cite{CY11}.

The order of the constraint sets in (\ref{eq:feasibility}) also requires some consideration. For matrix completion problems with two constraints, we can and do directly apply the Douglas--Rachford method to $C_1\cap C_2$,  with the reflection first performed with respect to the set $C_1$. For matrix completion problems with more than two constraints, we apply the Douglas--Rachford method to the product formulation of Example~\ref{ex:product}, with the reflection with respect to $D$ performed first. In this case, the solution is obtained by projecting onto $D$ and thus can be monitored by considering only a single product coordinate.

In non-convex applications, the problem formulation chosen often determines whether or not the Douglas--Rachford scheme can successfully solve the problem at hand  always, frequently or never, see also \cite{ABT13}.
 Hence, in the rest of this paper we focus on \emph{naive} or direct implementation of the Douglas-Rachford method while focusing on the choice of an appropriate model and the computation of the requisite projections/reflections. In a followup paper, we will look at more refined variants for our two capstone applications: to protein reconstruction and to Hadamard matrix problems.

\section{Convex Problems}\label{sec:conv}

We now look, in order, at positive-definite matrices and correlation matrices, doubly-stochastic matrices, and Euclidean distance matrices before discussing adjunction of noise.

\subsection{Positive semi-definite matrices}
Let $S_n$ denote the set of real $n\times n$ symmetric matrices. Recall that a matrix $A=(A_{ij})\in\R^{n\times n}$ is said to be \emph{positive semi-definite}  if
  \begin{equation}\label{eq:PSD}
   A\in S_n\text{ and }x^TAx\geq 0\text{ for all }x\in\R^n.
  \end{equation}
The set of all such matrices form a closed convex cone (see \cite[Ex.~1, Sec.~1.2]{BL06}), and shall be denoted by $S_+^n$. The \emph{Loewner partial order} is defined on $S_n$ by taking $A\succeq B$ if $A-B\in S_+^n$.
Recall that  a symmetric matrix is positive semi-definite if and only if all its eigenvalues are non-negative.

Let us consider the matrix completion problem where only some entries of the positive semi-definite matrix $A$ are known, and denote by $\Omega$ the location of these entries (i.e. $(i,j)\in\Omega$ if $A_{ij}$ is known). Without loss of generality, we may assume that $\Omega$ is symmetric in the sense that $(i,j)\in\Omega$ if and only if $(j,i)\in\Omega$. Consider the convex sets
 \begin{align}\label{eq:corrconstraints}
  C_1 := \{X\in\R^{n\times n}|X_{ij}=A_{ij}\text{ for all }(i,j)\in\Omega\}, \quad C_2:=S_n^+.
 \end{align}
Then $X$ is a positive semi-definite matrix that completes $A$ if and only if $X\in C_1\cap C_2$.

The set $C_1$ is a closed affine subspace. Its projection is straightforward, and given pointwise by
 \begin{equation}\label{eq:C1nn}
  P_{C_1}(X)_{ij}=\left\{\begin{array}{cc}
                     A_{ij} & \text{if }(i,j)\in\Omega, \\
                     X_{ij} & \text{otherwise};
                    \end{array}\right.
 \end{equation}
for all $i,j=1,\dots,n$.

\begin{theorem}[{\cite[Th.~2.1]{H86}}]\label{th:h86}
 Let $X\in\R^{n\times n}$. Define $Y=(A+A^T)/2$ and let $Y=UP$ be a polar decomposition (see \cite[Th.~1.1]{H86}). Then
  \begin{equation}\label{eq:projPSD}
   P_{C_2}(X)= \frac{Y+P}{2}.
  \end{equation}
\end{theorem}

\begin{remark}
 For $X\in S_n$, $Y=X$ in the statement of Theorem~\ref{th:h86}. If this is the case, the computation of $P_{C_2}$ is also simplified.

 If the initial matrix is symmetric, then the corresponding Douglas--Rachford iterates are too. This condition can be easily satisfied. For instance, if $X\in\R^{n\times n}$ then the iterates can instead be computed starting at $P_{S_n}(X)=(X+X^T)/2$ or $XX^T\in S_n$.

 Of course, for symmetric iterates only the the upper (or lower) triangular matrix need be computed. \qede
\end{remark}

\begin{remark}
 The matrices $U$ and $P$ can also be easily obtained from a \emph{singular value decomposition} (see \cite[p.~205]{HJ85}). For if $Y=WSV^T$ is a singular value decomposition then
  $$P=VSV^T,\quad U=WV^T$$
  produces $P$ and $U$. \qede
\end{remark}

\begin{remark}[Positive definite matrices]
 Recall that a real symmetric $n\times n$ matrix is said to be \emph{positive definite} if the inequality in (\ref{eq:PSD}) holds strictly whenever $x\neq 0$. Denote the set of all such matrices by $S^n_{++}$. Since $S^n_{++}$ is not closed, the problem of positive definite matrix completion cannot be directly cast within this framework by setting $C_2:=S^n_{++}$.

 In practice, one might  wish to consider a closed convex subset of $S_{++}^n$. For example, one could instead define
  \begin{equation} \label{eq:epsPD}
   C_2:=\{X\in\R^{n\times n}|X^T=X, x^TXx\geq \epsilon\|x\|^2\text{ for all }x\in\R^n\},
  \end{equation}
 for some small $\epsilon>0$. Then (\ref{eq:epsPD}) is equivalent to requiring that the eigenvalues be not less than $\epsilon$. \qede
\end{remark}

One can apply our methods to finding semi-definite solutions to matrix Riccati equations \cite{ABM93}.

\subsubsection{Correlation matrices}
An important class of positive semi-definite matrices is the \emph{correlation matrices}. Given random variables $X_1,X_2,\dots,X_n$, the associated correlation matrix is an element of $[-1,1]^{n\times n}$ whose $ij$th entry is the correlation between variables $X_i$ and $X_j$. Since, any random variable perfectly correlates with itself, the entries along the main diagonal of any correlation matrix are all ones. Consequently,
 \begin{equation}\label{eq:corr}
  \{(i,i)|i=1,\dots,n\}\subseteq\Omega, \text{ and }A_{ii}=1\text{ for }i=1,\dots,n.
 \end{equation}
Moreover whenever (\ref{eq:corr}) holds, $A$ is necessarily contained in $[-1,1]^{n\times n}$. This is a consequence of the following inequality.

\begin{proposition}[{\cite[p.~398]{HJ85}}]
 Let $A=(a_{ij})\in S^n_+$. Then
  $$a_{ii}a_{jj}\geq a_{ij}^2.$$
\end{proposition}

Thus, if $A$ is an incomplete correlation matrix, without loss of generality we may assume that (\ref{eq:corr}) holds. In this case, $X$ is correlation matrix that completes $A$ if and only if $X\in C_1\cap C_2$, as defined in (\ref{eq:corrconstraints}).

\bigskip Consider now the problem of generating a random sample of correlation matrices. This is the case, for example, when one uses simulation to determine an unknown probability distribution \cite{AR13, PHB13}.\bigskip

The Douglas--Rachford method provides a heuristic for generating such a sample by applying the method to initial points chosen according to some probability distribution. In this case, the set of known indices, and their corresponding values, are
 $$\Omega=\{(i,i)|i=1,\dots,n\},\text{ and }A_{ii}=1\text{ for }i=1,\dots,n.$$

The distribution of the entries in 100000 matrices of size $5\times 5$ obtained from three different sets of choices of initial point distribution is shown in Figure~\ref{fig:correlation}.

\begin{figure}
  \centering
  \begin{subfigure}[b]{0.3\textwidth}
    \centering
    \includegraphics[width=\textwidth]{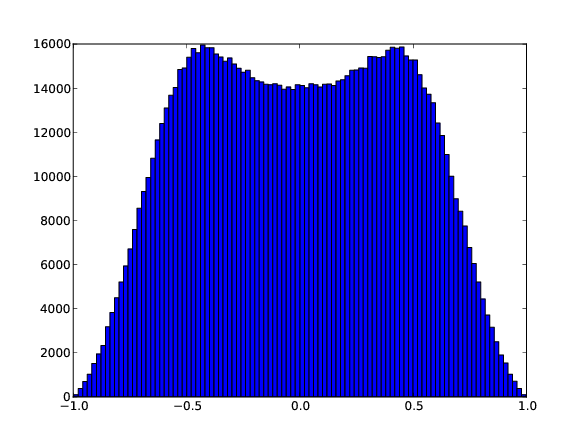}
    \caption{$X_0:=Y$.}
  \end{subfigure}
  \begin{subfigure}[b]{0.3\textwidth}
    \centering
    \includegraphics[width=\textwidth]{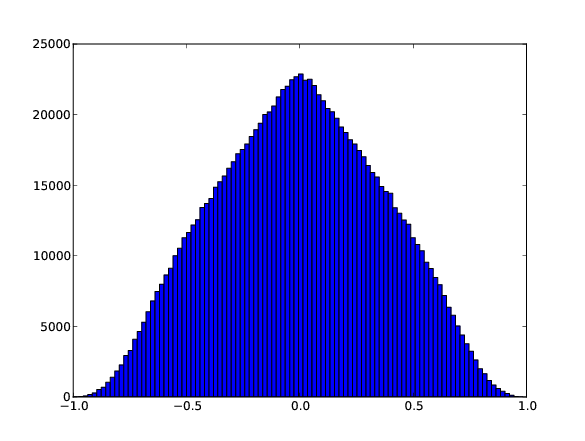}
    \caption{$X_0:=(Y+Y^T)/2$.}
  \end{subfigure}
  \begin{subfigure}[b]{0.3\textwidth}
    \centering
    \includegraphics[width=\textwidth]{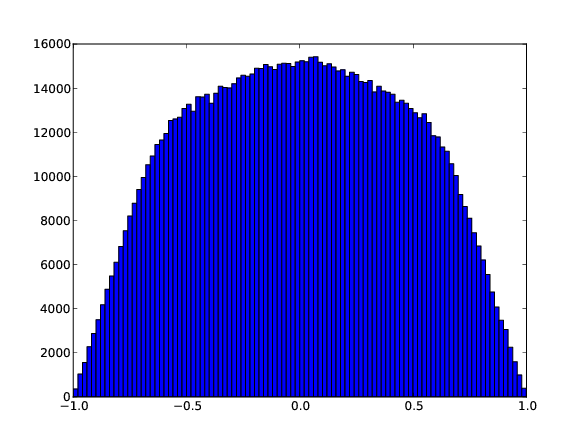}
    \caption{$X_0:=YY^T$.}
  \end{subfigure}
  \caption{Distribution of entries in the collections of correlation matrices generated by different  initialisations of the Douglas--Rachford algorithm. The initial point is $X_0$, and $Y$ is a random matrix in $[-1,1]^{5\times 5}$. Note $(Y+Y^T)/2,YY^T\in S_5$.}\label{fig:correlation}
\end{figure}

\subsection{Doubly stochastic matrices}\label{sec:doublystoc}
Recall that a matrix $A=(A_{ij})\in\R^{m\times n}$ is said to be \emph{doubly stochastic} if
 \begin{equation}
  \sum_{i=1}^mA_{ij}=\sum_{j=1}^nA_{ij}=1, A_{ij}\geq 0\text{ for }i=1,\dots,m \text{ and }j=1,\dots,n.
 \end{equation}
The set of all doubly stochastic matrices are known as the \emph{Birkhoff polytope}, and can be realised as the convex hull of the set of permutation matrices (see, for example, \cite[Th.~1.25]{BL06}).

Let us now consider the matrix completion problem where only some entries of a doubly stochastic matrix $A$ are known, and denote by $\Omega$ the location of these entries (i.e., $(i,j)\in\Omega$ if $A_{ij}$ is known). The set of all such candidates is given by 
\begin{align}
C_1 &:= \{X\in\R^{m\times n}|X_{ij}=A_{ij}\text{ for all }(i,j)\in\Omega\}, \\
\intertext{which is similar to (\ref{eq:corrconstraints}). The Birkhoff polytope may be expressed as the intersection of the three convex sets}
 C_2 &:= \left\{X\in\R^{m\times n}|\sum_{i=1}^mX_{ij}=1\text{ for }j=1,\dots,n\right\},\\
 C_3 &:= \left\{X\in\R^{m\times n}|\sum_{j=1}^nX_{ij}=1\text{ for }i=1,\dots,m\right\},\\
 C_4 &:= \{X\in\R^{m\times n}|X_{ij}\geq 0\text{ for }i=1,\dots,m\text{ and }j=1,\dots,n\}.
\end{align}
Then $X$ is a double stochastic matrix that completes $A$ if and only if $X\in C_1\cap C_2\cap C_3\cap C_4.$

As in (\ref{eq:C1nn}), the set $C_1$ is a closed affine subspace whose projection is given pointwise by
 \begin{equation}\label{eq:C1mn}
 P_{C_1}(X)_{ij}=\left\{\begin{array}{cc}
                     A_{ij} & \text{if }(i,j)\in\Omega, \\
                     X_{ij} & \text{otherwise};
                    \end{array}\right.
 \end{equation}
for all $i=1,\dots,m$ and $j=1,\dots,n$.

The projection onto $C_2$ (resp. $C_3$) can be easily computed by applying the following proposition row-wise (resp. column-wise).

\begin{proposition}
 Let $S:=\{x\in\R^m|\sum_{i=1}^m x_i=1\}$. For any $x\in\R^m$,
  $$P_S(x)= x+ \frac{1}{m}\left(1-\sum_{i=1}^mx_i\right)e, \text{ where }e=[1,1,\dots,1]^T.$$
\end{proposition}
\begin{proof}
 Since $S=\{x\in\R^n|\langle e,x\rangle=1\}$, the result follows from the standard formula for the orthogonal projection onto a hyperplane (see, for example, \cite[Sec.~4.2.1]{ER11}).
\end{proof}

The projection of $A$ onto $C_4$ is given pointwise by
 $$P_{C_4}(A)_{ij}=\max\{0,A_{ij}\},$$
for $i=1,\dots,m$ and $j=1,\dots,n.$

\begin{remark}
 One can also address the problem of singly-stochastic matrix completion. The problem of row (resp. column) stochastic matrix completion is formulated by dropping constraint $C_3$ (resp. $C_2$). \qede
\end{remark}

\subsection{Euclidean distance matrices}\label{ssec:euc}

A matrix $D=(D_{ij})\in\R^{n\times n}$ is said to be a \emph{distance matrix} if
$$D_{ij}=D_{ji}=\left\{\begin{array}{cl}
=0, &i=j,\\
\geq 0, &i\neq j;\end{array}\right. \text{ for }i,j=1,\dots,n.$$
Furthermore, $D$ is called a \emph{Euclidean distance matrix} (EDM) if there are points $p_1,\ldots, p_n\in \R^r$ (with $r\leq n$) such that
\begin{equation}\label{eq:EDM}
D_{ij}=\|p_i-p_j\|^2 \quad\text{for } i,j=1,\ldots,n.
\end{equation}
If~\eqref{eq:EDM} holds for a set of points in $\R^r$ then $D$ is said to be \emph{embeddable} in $\R^r$. If $D$ is embeddable in $\R^r$ but not in $\R^{r-1}$, then it is said to be \emph{irreducibly embeddable} in $\R^r$.

The following result by Hayden and Wells, based on Schoenberg's criterion~\cite[Th.~1]{S35}, provides a useful characterization of Euclidean distance matrices.

\begin{theorem}[{\cite[Th.~3.3]{HW88}}]\label{th:3.3}
Let $Q$ be the Householder matrix defined by
$$Q:=I-\frac{2vv^T}{v^Tv}, \text{ where }v=\left[1,1,\ldots,1,1+\sqrt{n}\right]^T\in\R^n.$$
Then, a distance matrix $D$ is a Euclidean distance matrix if and only if the ${(n-1)\times(n-1)}$ block $\widehat{D}$ in
\begin{equation}\label{eq:Q-DQ}
Q(-D)Q=\left[\begin{array}{cc}
\widehat{D} & d\\
d^T & \delta\end{array}\right]
\end{equation}
is positive semidefinite. In this case, $D$ is irreducibly embeddable in $\R^r$ where $r=\rank(\widehat{D})\leq n-1$.
\end{theorem}

\begin{remark}
As a consequence of Theorem~\ref{th:3.3}, the set of Euclidean distance matrices is convex.  \qede
\end{remark}

Let us consider now the matrix completion problem where only some entries of a Euclidean distance matrix $D$ are known, and denote by $\Omega$ the location of these entries (i.e., $(i,j)\in\Omega$ if $D_{ij}$ is known). Without loss of generality we assume $D$ and $\Omega$ to be symmetric. Consider the convex sets
\begin{align}
C_1:=&\big\{X\in\R^{n\times n}\mid X\text{ is a distance matrix}, X_{ij}=D_{ij}\text{ for all }(i,j)\in\Omega \big\},\label{eq:EDM1}\\
C_2:=&\big\{X\in\R^{n\times n}\mid \widehat{X}\succeq 0\text{ where }\widehat{X} \text{ is the block in }Q(-X)Q \text{ in }\eqref{eq:Q-DQ}\big\}\label{eq:EDM2}
\end{align}
Then $X$ is a Euclidean distance matrix that completes $D$ if and only if $X\in C_1\cap C_2$.

The projection of any symmetric matrix $A=(A_{ij})\in\R^{n\times n}$ onto $C_1$ can be easily computed:

\begin{equation}\label{eq:EDM_PC1}
P_{C_1}(A)=\left\{\begin{array}{cl}
0,&\text{if } i=j,\\
D_{ij},&\text{if }(i,j)\in\Omega,\\
\max\{0,A_{ij}\},&\text{otherwise;}\end{array}\right.
\end{equation}
The projection of $A$ onto $C_2$ is the unique solution to the problem
$$\min_{X\in C_2} \|A-X\|_F.$$
If we denote
$$Q(-A)Q=\left[\begin{array}{cc}
\widehat{A} & a\\
a^T & \alpha\end{array}\right]\quad\text{and}\quad Q(-X)Q=\left[\begin{array}{cc}
\widehat{X} & x\\
x^T & \lambda\end{array}\right],$$
then
\begin{align*}
\min_{X\in C_2} \|A-X\|_F&=\min_{X\in C_2} \|Q(A-X)Q\|_F=\min_{X\in C_2} \|Q(-A)Q-Q(-X)Q\|_F\\
&=\min_{x\in\R^n,\lambda\in\R\atop\widehat{X}=\widehat{X}^T, \widehat{X}\succeq 0} \left\|\begin{array}{cc}
\widehat{A}-\widehat{X} & a-x\\
(a-x)^T & (\alpha-\lambda)\end{array}\right\|_F.\end{align*}
A consequence of~\cite[Th.~2.1]{HW88} is that
the unique best approximation
is given by
$$\left[\begin{array}{cc}
U\Lambda_+U^T & a\\
a^T & \alpha\end{array}\right]$$
where $U\Lambda U^T=\widehat{A}$ is the spectral decomposition (see \cite[p.116]{HW88}) of $\widehat{A}$, with $\Lambda=\diag(\lambda_1,\ldots,\lambda_{n-1})$, and $\Lambda_+=\diag(\max\{0,\lambda_1\},\ldots,\max\{0,\lambda_{n-1}\})$.
Therefore,
\begin{equation}\label{eq:EDM_PC2} P_{C_2}(A)=-Q\left[\begin{array}{cc}
U\Lambda_+U^T & a\\
a^T & \alpha\end{array}\right]Q.
\end{equation}

\subsubsection{Noise}

In many practical situations the distances that are initially known have some errors in their measurements, and the Euclidean matrix completion problem may not even have a solution. In these situations, a model that allows errors in the distances needs to be considered.

Given some error $\eps\geq 0$, consider the convex set
\begin{align}
C_1^\eps:=\big\{X\in\R^{n\times n}\mid & \,X\text{ is a distance matrix}\nonumber\\
&\text{ and } |X_{ij}-D_{ij}|\leq\eps\text{ for all }(i,j)\in\Omega \big\}.\label{eq:EDM1eps}
\end{align}
Notice that $C_1^0=C_1$. The projection of any symmetric matrix $A=(A_{ij})\in\R^{n\times n}$ onto $C_1^\eps$ can be easily computed:

\begin{equation}
P_{C_1^\eps}(A)=\left\{\begin{array}{cl}
0,&\text{if } i=j,\\
D_{ij}+\eps,&\text{if }(i,j)\in\Omega\text{ and }A_{ij}>D_{ij}+\eps,\\
\max\{0,D_{ij}-\eps\},&\text{if }(i,j)\in\Omega\text{ and }A_{ij}<D_{ij}-\eps,\\
\max\{0,A_{ij}\},&\text{otherwise.}\end{array}\right.
\end{equation}
This model could be easily modified to include a different upper and lower bound on each distance $D_{ij}$ for $(i,j)\in\Omega$.

\section{Non-convex Problems}\label{sec:nonconv}

We now turn to the more difficult case of non-convex matrix completion problems.

\subsection{Low-rank matrices}

It many practical scenarios, one would like to recover a matrix that is known to be low-rank from only a subset of its entries. This is the case, for example, in various compressed sensing applications \cite{BL11}. The main problem here is that the low-rank constraint makes the problem non-convex. For example, if we consider
$$S:=\big\{A\in\R^{2\times 2}\big|\rank(A)\leq 1\big\},$$
then
$$\left[\begin{array}{cc}
1 & 0\\
0 & 0\end{array}\right],
\left[\begin{array}{cc}
0 & 0\\
0 & 1\end{array}\right]\in S,$$
but for all $\lambda\in(0,1)$,
$$\lambda\left[\begin{array}{cc}
1 & 0\\
0 & 0\end{array}\right]+(1-\lambda)
\left[\begin{array}{cc}
0 & 0\\
0 & 1\end{array}\right]\not\in S.$$

\subsubsection{Relaxed rank constraints}

Let us consider the problem of finding a matrix of minimal rank, given that some of the entries are known. We define a \emph{relaxed rank constraint}
\begin{equation*}
 C_2^r:=\{X\in\R^{m\times n}|\rank(X)\leq r\}.
\end{equation*}
Then $X$ is a matrix completion of $A$ with rank at most $r$ if and only if $X\in C_1\cap C_2^r$.

The set of possible ranks of $A$ is finite and bounded above by $\min\{m,n\}$. Furthermore, $C_2^r\subseteq C_2^s$ for $r\leq s$. It follows that $X$ is a completion of $A$ with minimal rank if and only if
 $$X\in C_1\cap C_2^{r_0} \text{ and }X\not\in C_1\cap C_2^r\text{ for any }r<r_0.$$
In this case $\rank(X)=r_0$.

This suggests a \emph{binary search} heuristic for finding the rank of a matrix. For convenience, abbreviate the Douglas--Rachford method by DR, and denote by $P^{(r)}$ the relaxation
\begin{equation}
 \text{Find }x\in C_1\cap C_2^r.
\end{equation}
The iteration can now be implemented as shown as Algorithm \ref{alg:relrank}:

\medskip
\begin{algorithm}[H]\label{alg:relrank}
 \caption{Heuristic for minimum rank matrix completion.}
 %Define formatting
 \SetKwInOut{Input}{input}
 \SetKwInOut{Output}{output}
 %The algorithm
 \Input{$\Omega$, $A_{ij}$ for all $(i,j)\in\Omega$, $MaxIterations$}
 $r_{lb}:=0,r_{ub}:=\min\{m,n\},r:=\lfloor r_{ub}/2\rfloor$\;
 \While{$r_{lb}<r_{ub}$}{
   \uIf{DR solves $P^{(r)}$ within $MaxIterations$ iterations }{
     $r_{ub}:=r$\;
   }
   \Else{
     $r_{lb}:=r+1$\;
   }
   $r:=\lfloor (r_{lb}+r_{ub})/2 \rfloor$\;
 }
 \Output{$r$}
\end{algorithm}
\medskip

Of course there are many applicable variants on this idea. For instance, one could instead perform a ternary search.

\subsection{Low-rank Euclidean distance matrices}

In many situations, the Euclidean distance matrix $D$ that one aims to complete is known to be embeddable in $\R^r$, say with $r=3$. This is the case, for example, in the molecular conformation problem in which one would like to compute the relative atom positions within a molecule. Nuclear magnetic resonance spectroscopy can be employed to measure short range interatomic distances (i.e. those less than 5--6\AA)\footnote{1\AA\ = $10^{-10}$ meters. The \AA\ stands for \emph{\AA ngstr\"om}.} without structural modification of the molecule (see~\cite{Yuen10}).

These types of problems are known as \emph{low-rank Euclidean distance matrix} problems. For any given positive integer $r$, we can modify the set $C_2$ in~\eqref{eq:EDM2} as follows
\begin{align*}
C_2^r:=\big\{X\in\R^{n\times n}\mid& \,\widehat{X}\succeq 0\text{ where }\widehat{X} \text{ is the block}\\
&\text{ in }Q(-X)Q \text{ in }\eqref{eq:Q-DQ}\text{ and }\rank(\widehat{X})\leq r\big\}.
\end{align*}
Unfortunately, as noted in~\cite[\S5.3]{GM89}, the set $C_2^r$ is no longer convex unless $r\geq n-1$ (in which case the rank condition is always satisfied and $C_2^r=C_2$). Nevertheless, a projection\footnote{Since $C_2^r$ is not convex, the projection need not be unique.} of any symmetric matrix $A$ onto $C_2^r$ can be easily computed. Indeed, let us assume without loss of generality that the eigenvalues $\lambda_1,\ldots,\lambda_{n-1}$ of the submatrix $\widehat{X}$ are given in nondecreasing order $\lambda_1\leq\lambda_2\leq\ldots\leq\lambda_{n-1}$ in the spectral decomposition $\widehat{X}=U\Lambda U^T$, where $\Lambda=\diag(\lambda_1,\ldots,\lambda_{n-1})$. Then, $P_{C_2^r}(X)$ can be computed as in~\eqref{eq:EDM_PC2} but with $\Lambda_+$ replaced by $$\Lambda_+^r:=\diag(0,\ldots,0,\max\{0,\lambda_{n-r}\},\ldots,\max\{0,\lambda_{n-1}\}).$$

\subsection{Protein reconstruction} \label{sec:prot}
Once more, despite the absence of convexity in one of the constraints, we have observed the Douglas--Rachford algorithm to converge. Computational experiments have been performed on various protein molecules obtained from the \emph{RCSB Protein Data  Bank}.\footnote{Available at  \url{http://www.rcsb.org/pdb/}.} The complete structure of these proteins is contained in the respective data files as a list of points in $\R^3$, each representing an individual atom. The corresponding complete Euclidean distance matrix can then be computed using (\ref{eq:EDM}). A realistic partial Euclidean distance matrix is then obtained by removing all entries which correspond to distances greater than 6\AA. From this partial matrix, we seek to reconstruct the molecular conformation.

In Algorithm \ref{algor:DRprotein} we give details regarding our \emph{Python} implementation for finding the distance matrix and in Algorithm \ref{alg:points} we reconstruct the positions from the matrix completion.
 \medskip

\begin{algorithm}[H]
 \caption{Douglas--Rachford component of our Python implementation.}\label{algor:DRprotein}
 \SetAlgoLined
 %Define formatting
 \SetKwInOut{Input}{input}
 \SetKwInOut{Initialize}{initialize}
 \SetKwInOut{Output}{output}
 \Input{$D\in\R^{n\times n}$ (the partial Euclidean distance matrix)}
 $\epsilon:=0.1,\,r:=3,\,N:=5000,\,k:=0$\;
 $X:=(Y+Y^T)/2\in S_n$ for random $Y\in[-1,1]^{n\times n}$\;
 \While{$k\leq N$}{
  $X:=T_{C_1^\epsilon,C_2^r}X$\;
  $k:=k+1$\;
 }
 \Output{$X$ (the reconstructed Euclidean distance matrix)}
\end{algorithm}

\bigskip

The quality of the solution is then assessed using various error measurements. The \emph{relative error}, reported in \emph{decibels} (dB), is given by
  $$\text{Relative error}:=10\log_{10}\left(\frac{\|P_{C_2^r}P_{C_1^\epsilon}X_N-P_{C_1^\epsilon}X_N\|_F^2}{\|P_{C_1^\epsilon}X_N\|_F^2}\right), \text{ where }\epsilon=0.1,r=3.$$

\begin{algorithm}[H]\label{alg:points}
 \caption{Converting a Euclidean distance matrix to points in $\R^n$ (see {\cite[Sec.~5.12]{BV04}}).}

 \SetKwInOut{Input}{input}
 \SetKwInOut{Output}{output}
 \Input{$X\in\R^{n\times n}$ (the reconstructed distances matrix)}
 $L:=I-ee^T/n$ where $e=(1,1,\dots,1)^T$\;
 $\tau:=-\frac{1}{2}LDL$\;
 $USV^T:=SingularValueDecomposition(\tau)$\;
 $Z:=$first $n$ columns of $U\sqrt{S}$\;
 $p_i:=i$th row of $Z$ for $i=1,2,\dots,n$\;
 \Output{$p_1,p_2,\dots,p_n$ (positions of the points in $\R^n$)}
 \end{algorithm}

\medskip

Let $p_1,p_2,\dots,p_n\in\R^3$ denote the positions of the $n$ atoms obtained from the distance matrix $X_N$, and let $p_1^{true},p_2^{true},\dots,p_n^{true}$ denote the true positions of the $n$ atoms (both relative to the same coordinate system). It is possible for both sets of points to represent the same molecular conformation without occupying the same positions in space. Thus, to compare the two sets, a \emph{Procrustes analysis} is performed.\footnote{This can be performed, for example, using build-in MATLAB functions.} That is, we (collectively) translate, rotate and reflect the points $p_1,p_2,\dots,p_n$ to obtain the point $\hat p_1,\hat p_2,\dots,\hat p_n$ which minimize the least squared error to the true positions.

Using the fitted points, we compute the \emph{root-mean-square error (RMSE)} defined by
 $$\text{RMSE}:=\sqrt{\frac{1}{\text{\# of atoms}}\sum_{i=1}^m\|\hat p_i-p_i^{true}\|_2^2},$$
and the \emph{maximum error} defined by
 $$\text{Max Error}:=\max_{1\leq i\leq m}\|\hat p_i-p_i^{true}\|_2.$$

\begin{table}[htbp]
 \begin{center}
 \caption{Six Proteins: average (maximum) errors from five replications (5000 iterations).}\label{table1}
 \begin{tabular}{ccccc}\hline
   Protein & \# Atoms & Relative Error (dB) & RMSE & Max Error \\ \hline
   1PTQ    & 404  & -83.6 (-83.7) & 0.0200 (0.0219)  & 0.0802 (0.0923)  \\
   1HOE    & 581  & -72.7 (-69.3) & 0.191 (0.257)    & 2.88 (5.49)      \\
   1LFB    & 641  & -47.6 (-45.3) & 3.24 (3.53)      & 21.7 (24.0)      \\
   1PHT    & 988  & -60.5 (-58.1) & 1.03 (1.18)      & 12.7 (13.8)      \\
   1POA    & 1067 & -49.3 (-48.1) & 34.1 (34.3)      & 81.9 (87.6)       \\
   1AX8    & 1074 & -46.7 (-43.5) & 9.69 (10.36)     & 58.6 (62.6)       \\ \hline
 \end{tabular}
 \end{center}
\end{table}

Our computational results  are summarized in Table~\ref{table1}. An animation of the algorithm at work constructing the protein  1PTQ can be viewed at  \url{http://carma.newcastle.edu.au/DRmethods/1PTQ.html}. We next make some general comments regarding the performance of our method.

\begin{itemize}
 \item 1PTQ and 1HOE, the two proteins with less than 600 atoms, could be reliability reconstructed to within a small error. A visual comparison of the reconstructed and original molecules match well -- they are indistinguishable. See Figures~\ref{fig:1PTQ} and \ref{fig:proteins}.
 \item The reconstructions of 1LFB and 1PHT, the next two smallest proteins examined, were both satisfactory although not as good as their smaller counterparts. A careful comparison of the original and reconstructed images in Figure~\ref{fig:proteins}, shows that a large proportion of the proteins have been faithfully reconstructed, although some finer details are missing. For instance, one should look at the top right corners of the 1PHT images.
 \item The reconstructions of 1POA and 1AX8, the largest two proteins examined, were poor. The  images of the reconstructed proteins show that some bond lengths are abnormally large. We discuss possible approaches to this issue in Remarks~\ref{remark:protein1} and \ref{remark:protein2}.
 \item Some alternative approaches to protein reconstruction are reported in \cite{ALA12}. Three are:
  \begin{itemize}
   \item A ``build-up" algorithm  placing atoms sequentially (\emph{Buildup}).
   \item A classical multidimensional scaling approach (\emph{CMDSCALE}).
   \item Global continuation on Gaussian smoothing of the error function (\emph{DGSOL}).
  \end{itemize}
  For 1PTQ and HOE, the RMS error of the Douglas--Rachford reconstruction was slightly smaller than the reconstruction obtained from either the buildup algorithm or CMDSCALE. For 1LFB and 1PHT the RMS errors were comparable, and for 1POA and 1AXE they still had the same order of magnitude. DGSOL performed better than all three approaches (Douglas--Rachford, Buildup and CMSCALE).
 \item For the proteins examined, computational times for the full 5000 iterations, except for 1POA,  ran anywhere from 6 to 18 hours. This time is mostly consumed by eigen-decompositions performed as part of computing $P_{C_2^r}$ and could perhaps be dramatically reduced by using a cheaper approximate projection.  For 1POA we used up to 50 hours for a full reconstruction.
\end{itemize}

\begin{sidewaysfigure}[h]
  \begin{center}

  \begin{subfigure}{0.19\textwidth}
  \includegraphics[width=\textwidth]{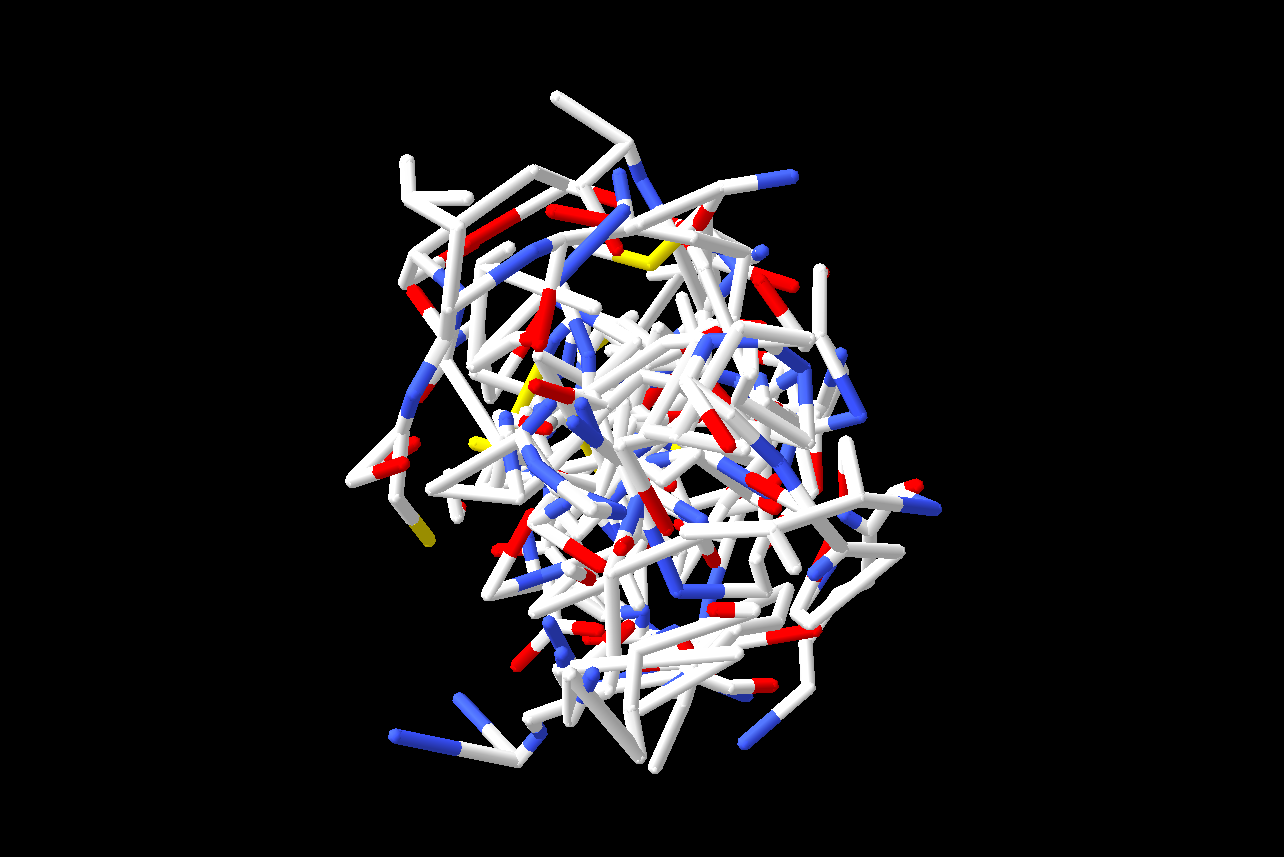}\caption{500 steps, -25 dB.}
  \end{subfigure}
  \begin{subfigure}{0.19\textwidth}
  \includegraphics[width=\textwidth]{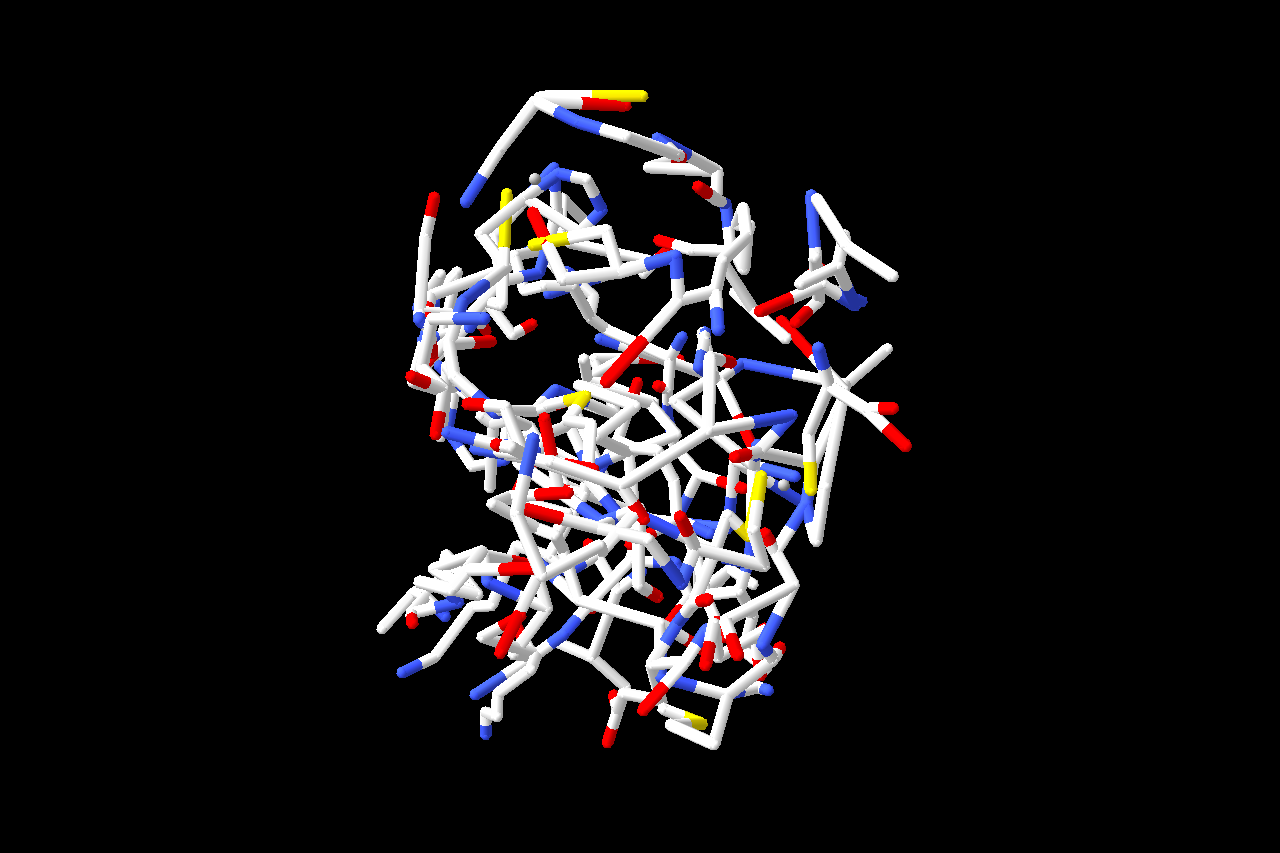}\caption{1,000 steps, -30 dB.}
  \end{subfigure}
  \begin{subfigure}{0.19\textwidth}
  \includegraphics[width=\textwidth]{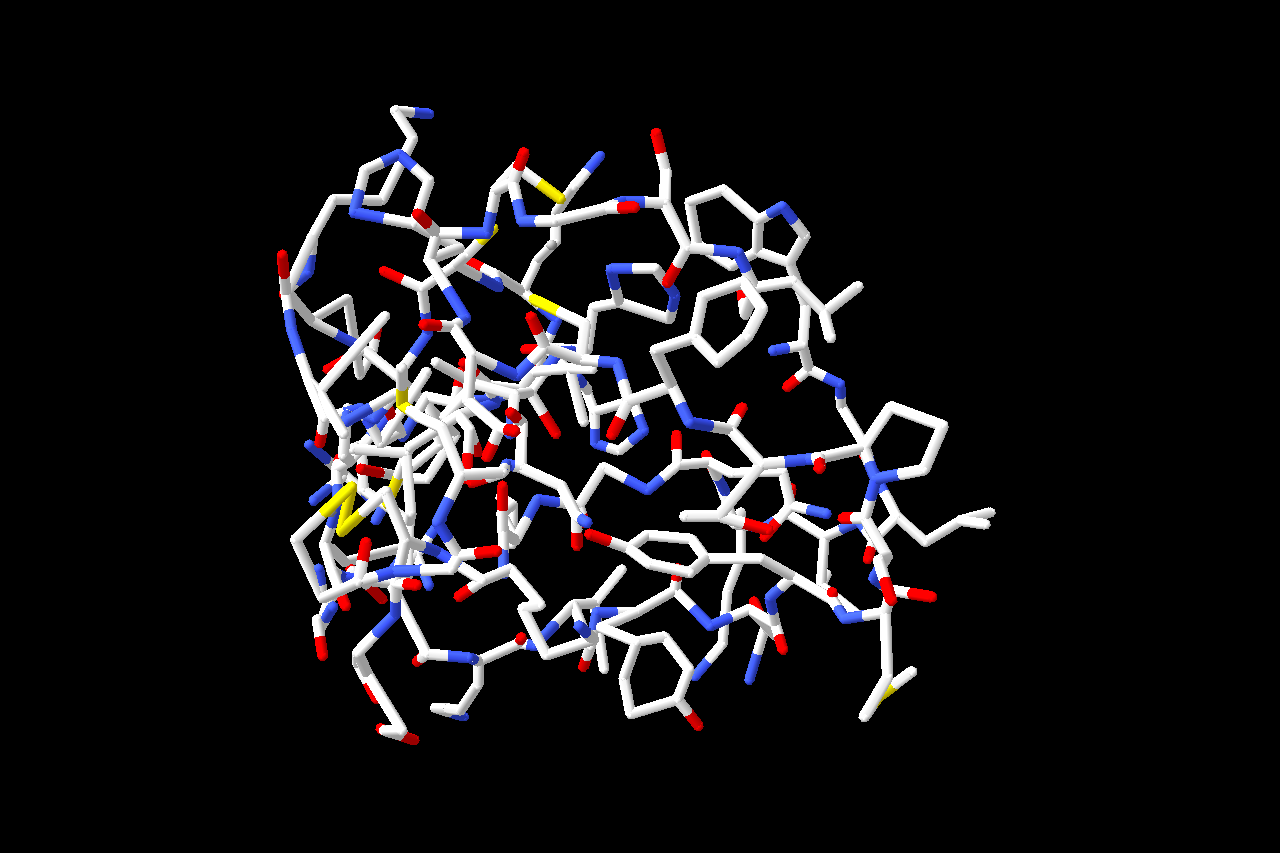}\caption{1,500 steps, -37 dB.}
  \end{subfigure}
  \begin{subfigure}{0.19\textwidth}
  \includegraphics[width=\textwidth]{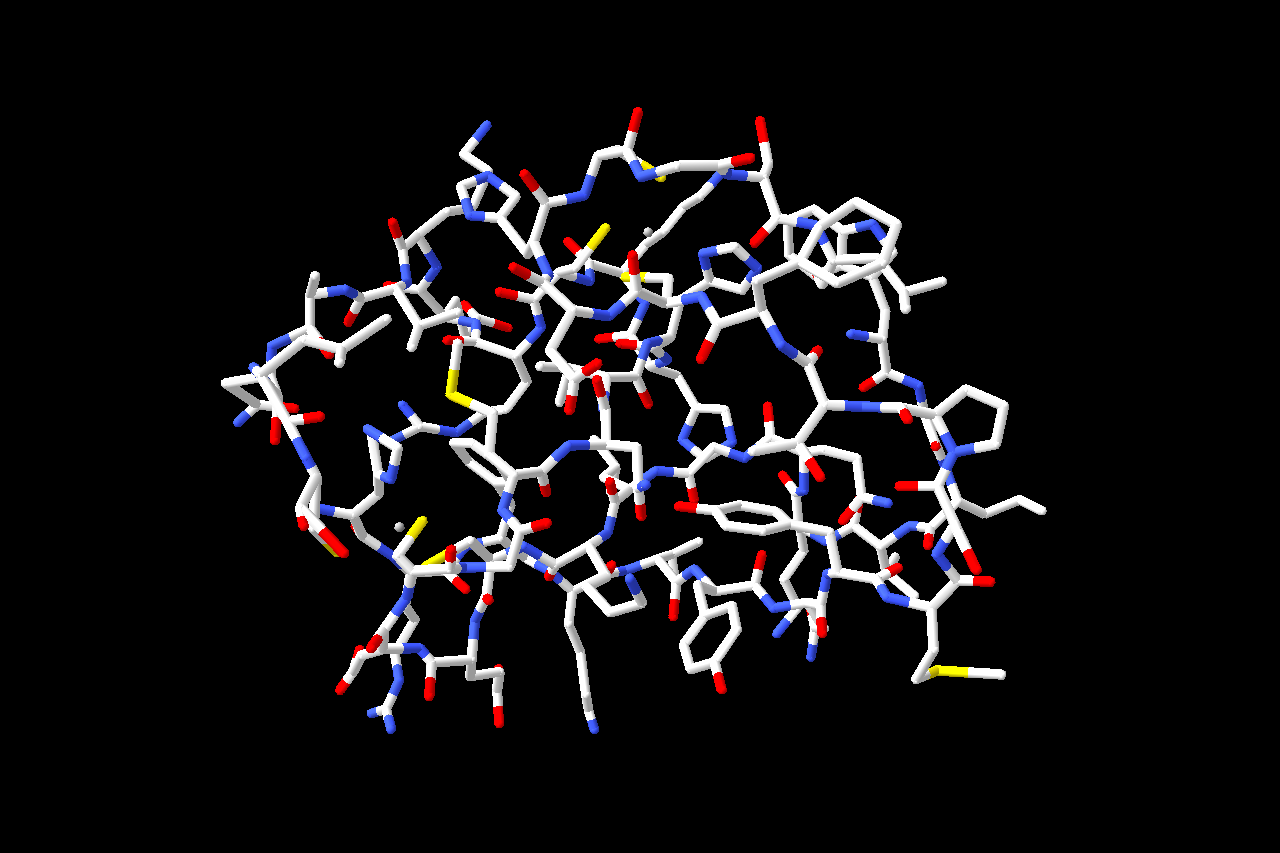}\caption{2,000 steps, -51 dB.}
  \end{subfigure}
  \begin{subfigure}{0.19\textwidth}
  \includegraphics[width=\textwidth]{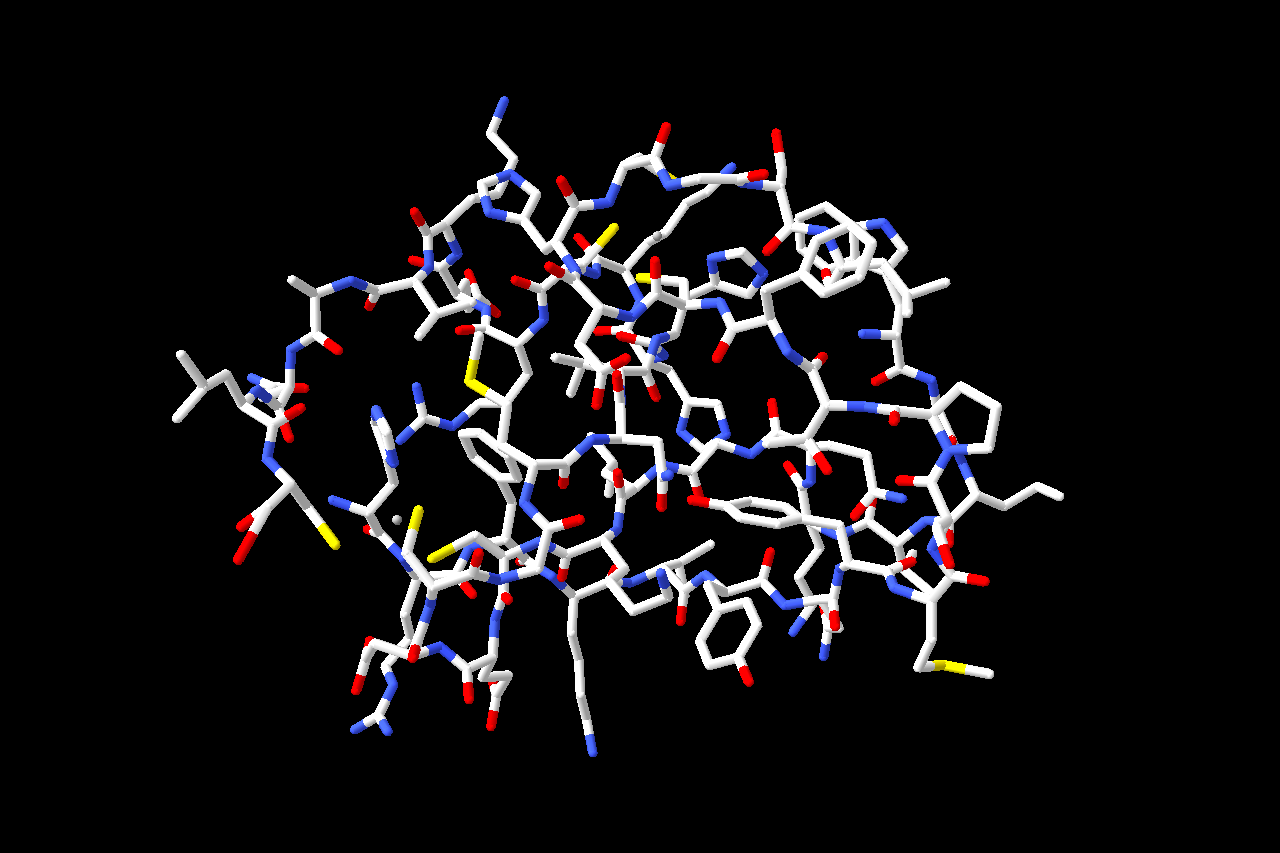}\caption{5,000 steps, -84 dB.}
  \end{subfigure}

  \bigskip

   \begin{subfigure}{0.19\textwidth}
  \includegraphics[width=\textwidth]{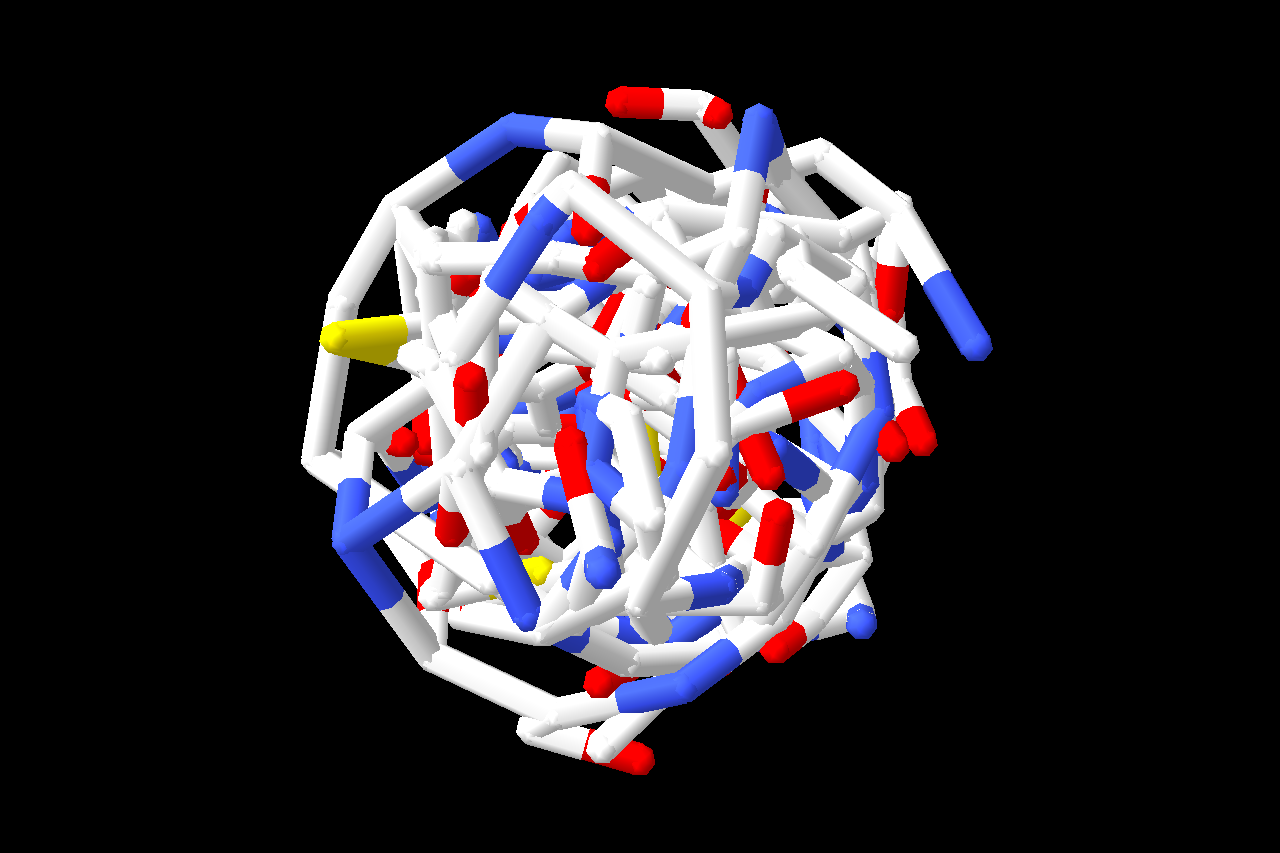}\caption{500 steps, -22 dB.}
  \end{subfigure}
  \begin{subfigure}{0.19\textwidth}
  \includegraphics[width=\textwidth]{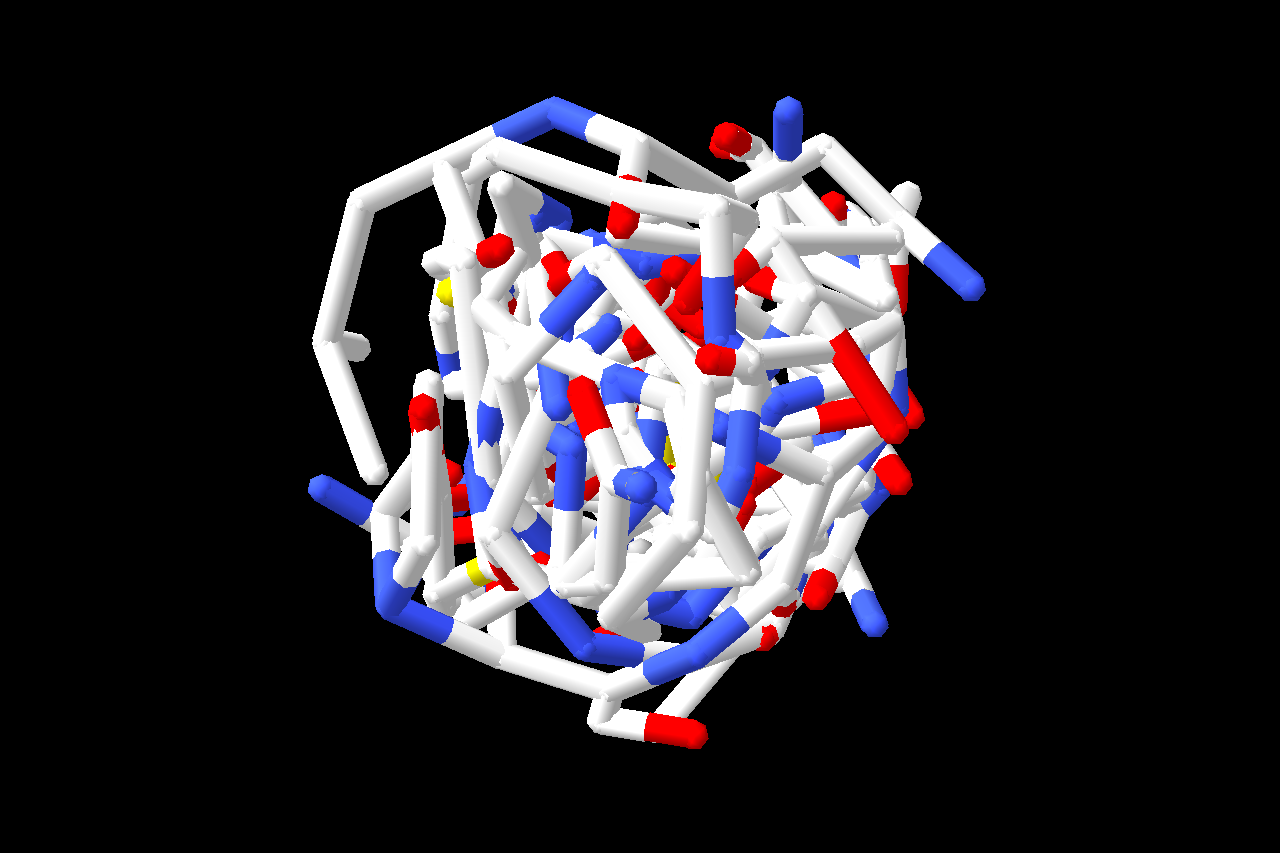}\caption{1,000 steps, -24 dB.}
  \end{subfigure}
  \begin{subfigure}{0.19\textwidth}
  \includegraphics[width=\textwidth]{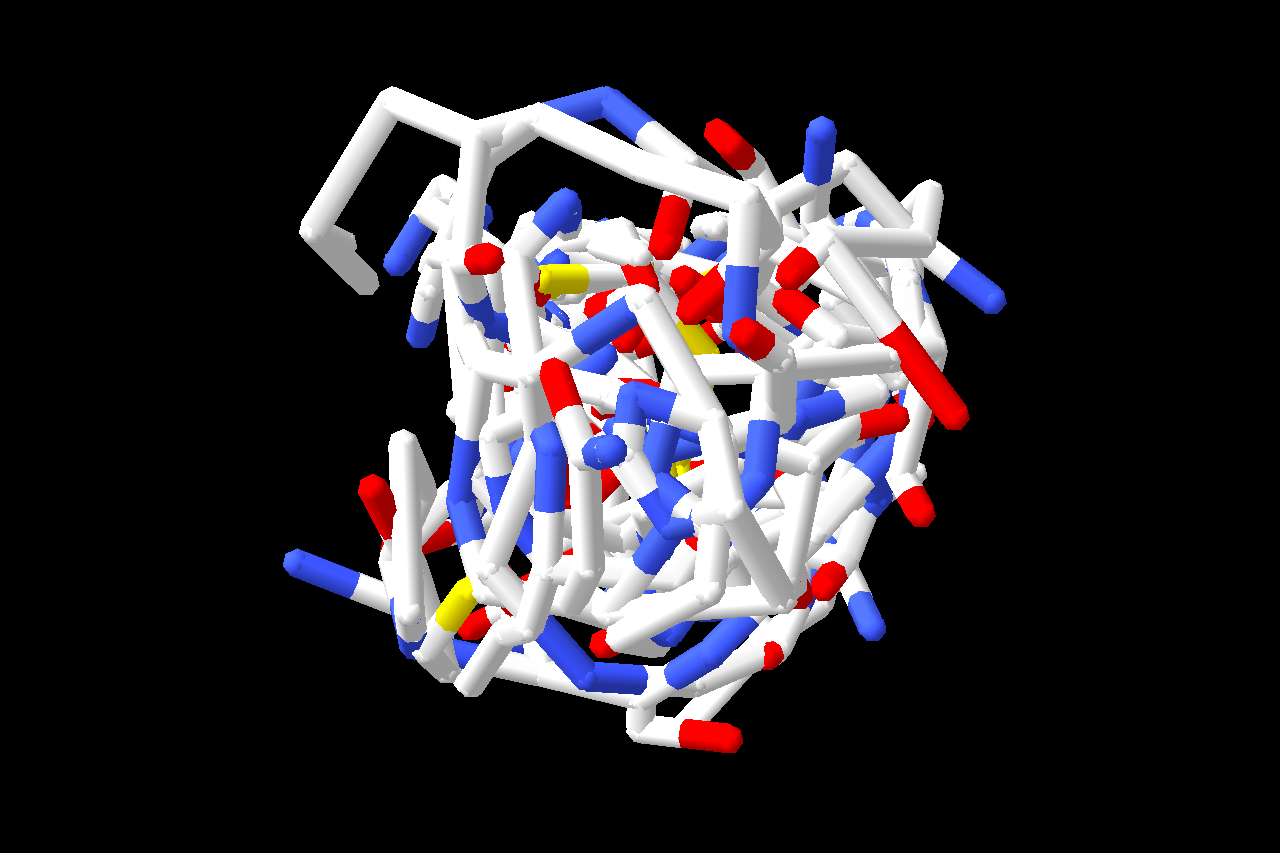}\caption{1,500 steps, -25 dB.}
  \end{subfigure}
  \begin{subfigure}{0.19\textwidth}
  \includegraphics[width=\textwidth]{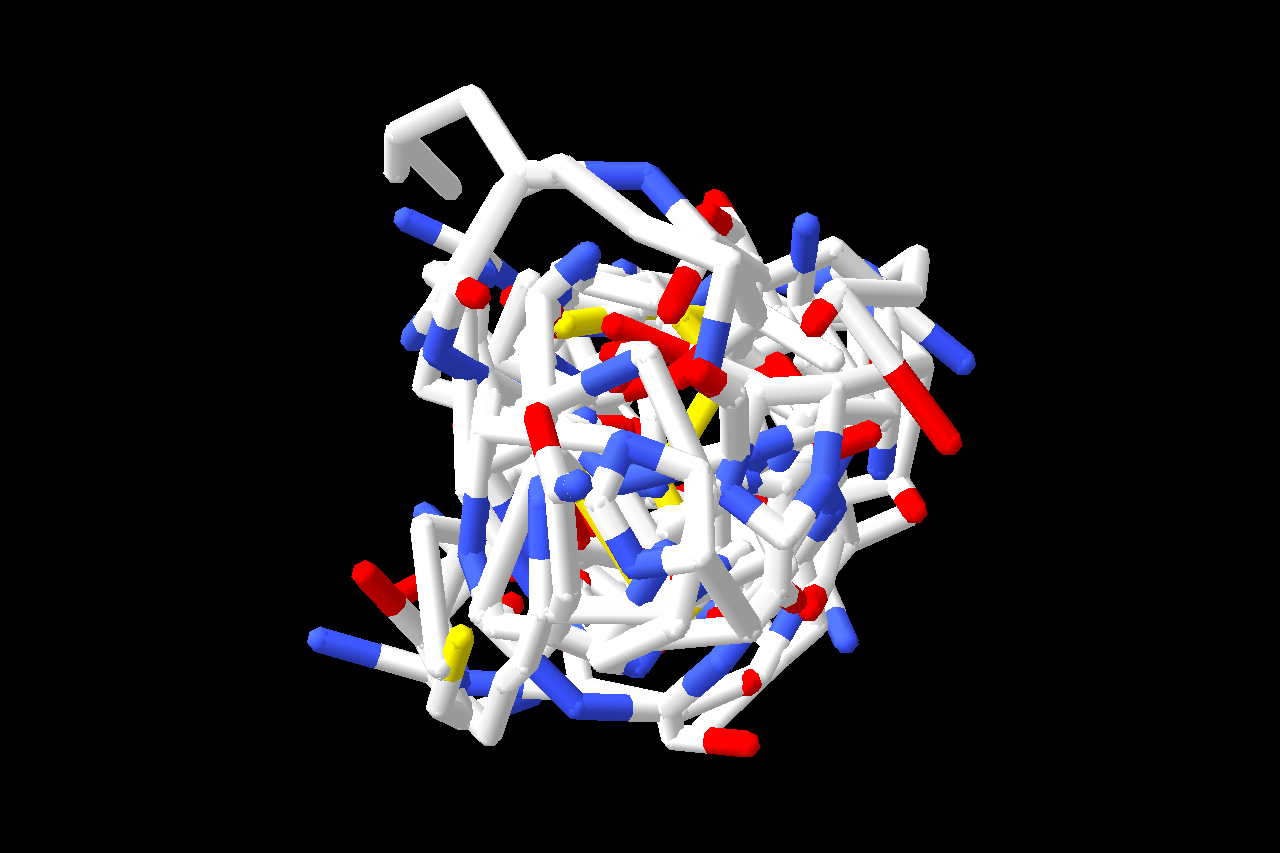}\caption{2,000 steps, -25 dB.}
  \end{subfigure}
  \begin{subfigure}{0.19\textwidth}
  \includegraphics[width=\textwidth]{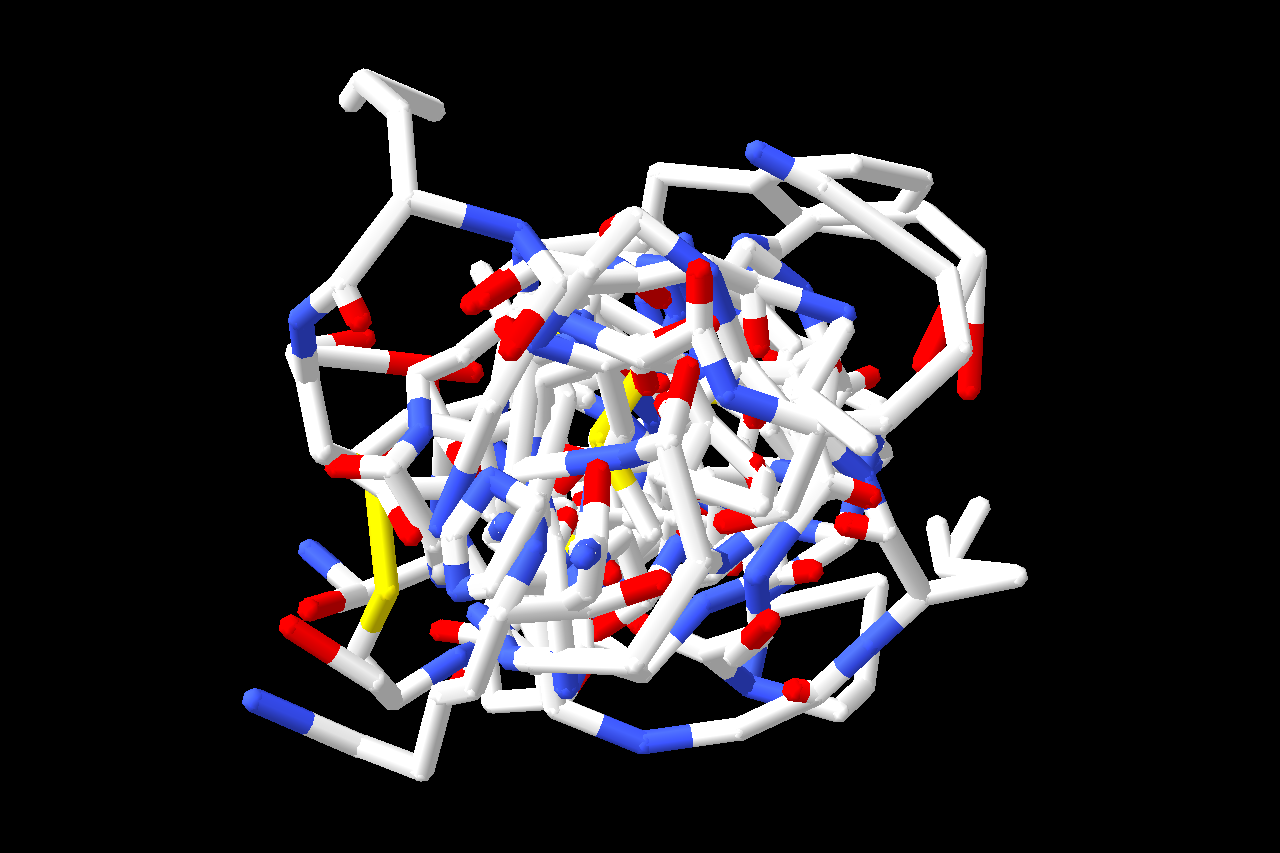}\caption{5,000 steps, -28 dB.}
  \end{subfigure}

  \bigskip

  \begin{subfigure}{0.19\textwidth}
  \includegraphics[width=\textwidth]{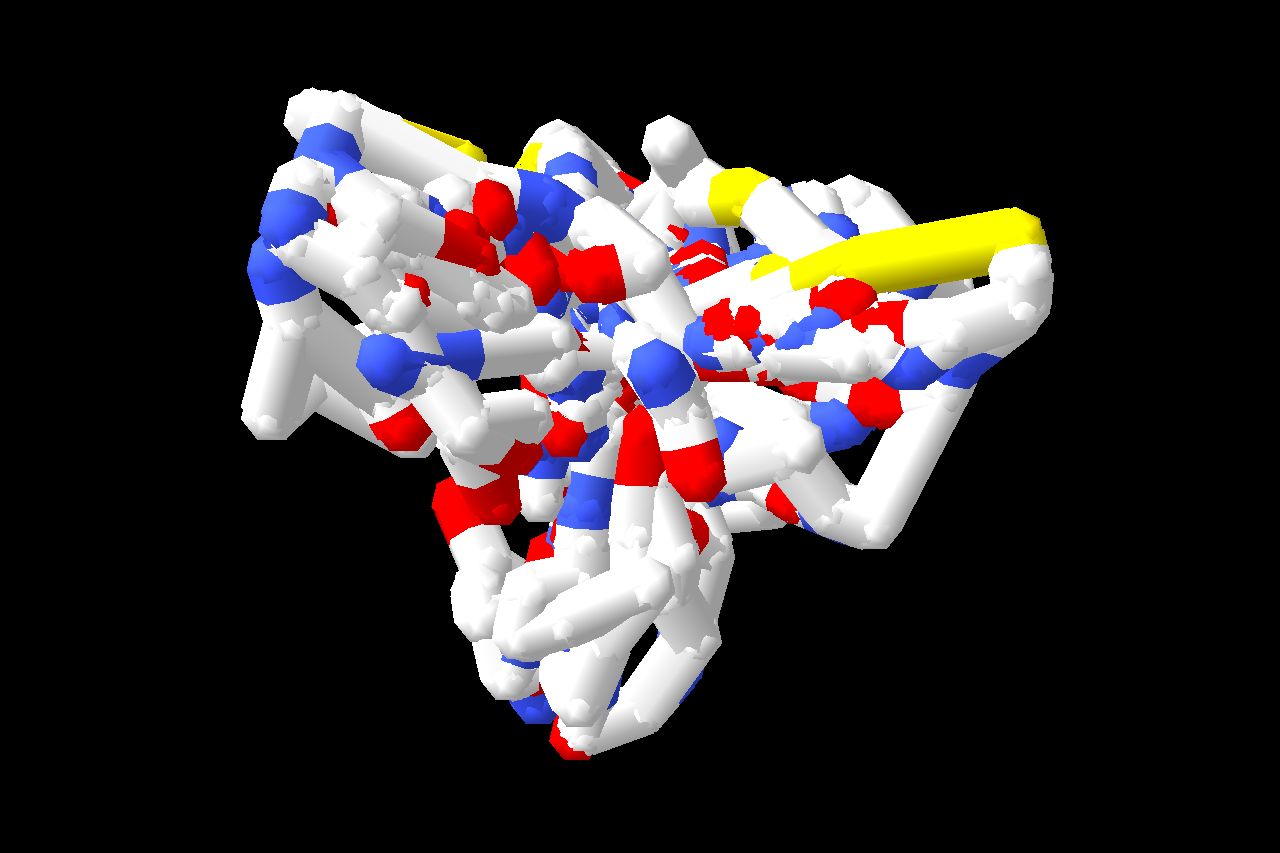}\caption{Before.}
  \end{subfigure}
  \qquad
  \begin{subfigure}{0.19\textwidth}
  \includegraphics[width=\textwidth]{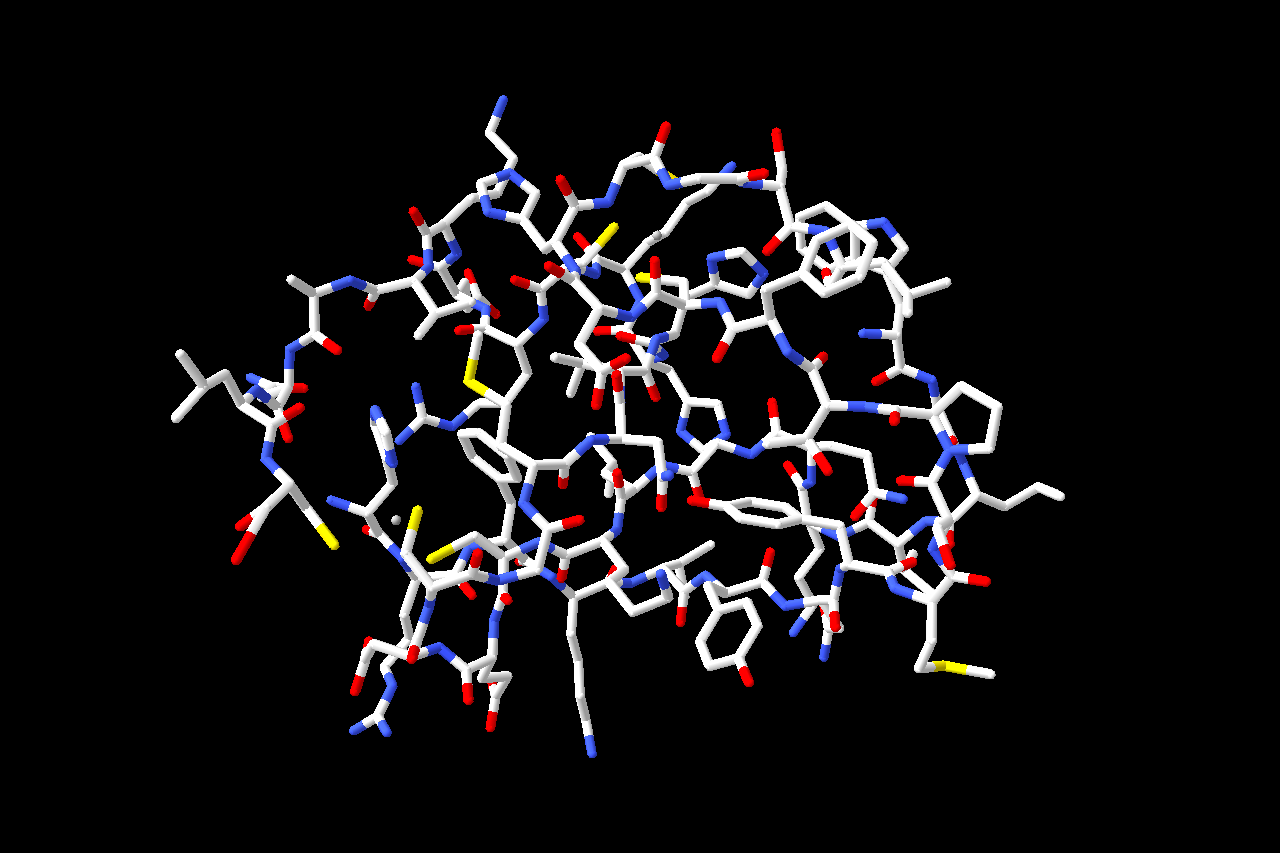}\caption{Original.}
  \end{subfigure}
  \qquad
  \begin{subfigure}{0.19\textwidth}
  \includegraphics[width=\textwidth]{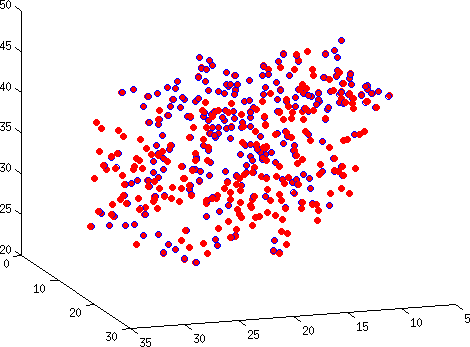}\caption{Atom positions.}
  \end{subfigure}

  \caption{Reconstructions (displayed in Swiss-PdbViewer) of the protein 1PTQ obtained from the Douglas--Rachford algorithm (a)--(e) and from the method of cyclic projections (f)--(j), together with their relative errors  after given numbers of steps. The protein prior to the reconstruction is shown in (k) and its actual structure in (l). Only interatomic distances below 6\AA\ have been used as input. This represents 14,370/162,812 distances (i.e. 8.83\% of the nonzero entries of the distance matrix). Entry (m) shows the positions of the original (resp. reconstructed) atoms in red (resp. blue) -- coincidence is frequent.} \label{fig:1PTQ}
  \end{center}
\end{sidewaysfigure}

\begin{figure}
  \begin{adjustwidth}{-0.5in}{-0.5in} %please remove adjustwidth --- centering figure for now...
  \begin{center}
    \caption{The five proteins not shown in Figure~\ref{fig:1PTQ}. The first column shows positions of  original (resp. reconstructed) atom in red (resp. blue), the second and third columns show the original protein and a reconstructioned instance (displayed in Swiss-PdbViewer),  as reported in Table~\ref{table1}.}\label{fig:proteins}
    \begin{tabular}{|c|c|c|c|} \hline
    Protein & Atom positions & Original & Reconstruction \\ \hline
    \rotatebox{90}{\hspace{32pt}1HOE}
              & \includegraphics[width=0.32\textwidth]{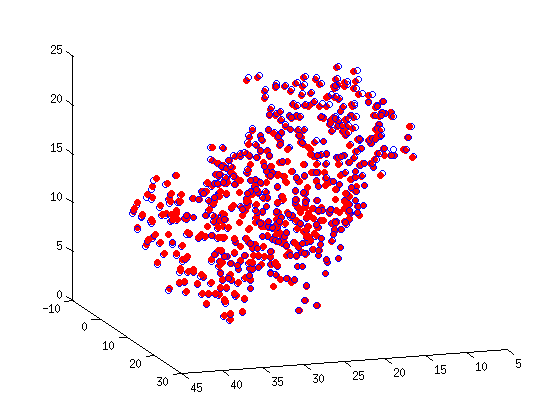}
              & \includegraphics[width=0.32\textwidth]{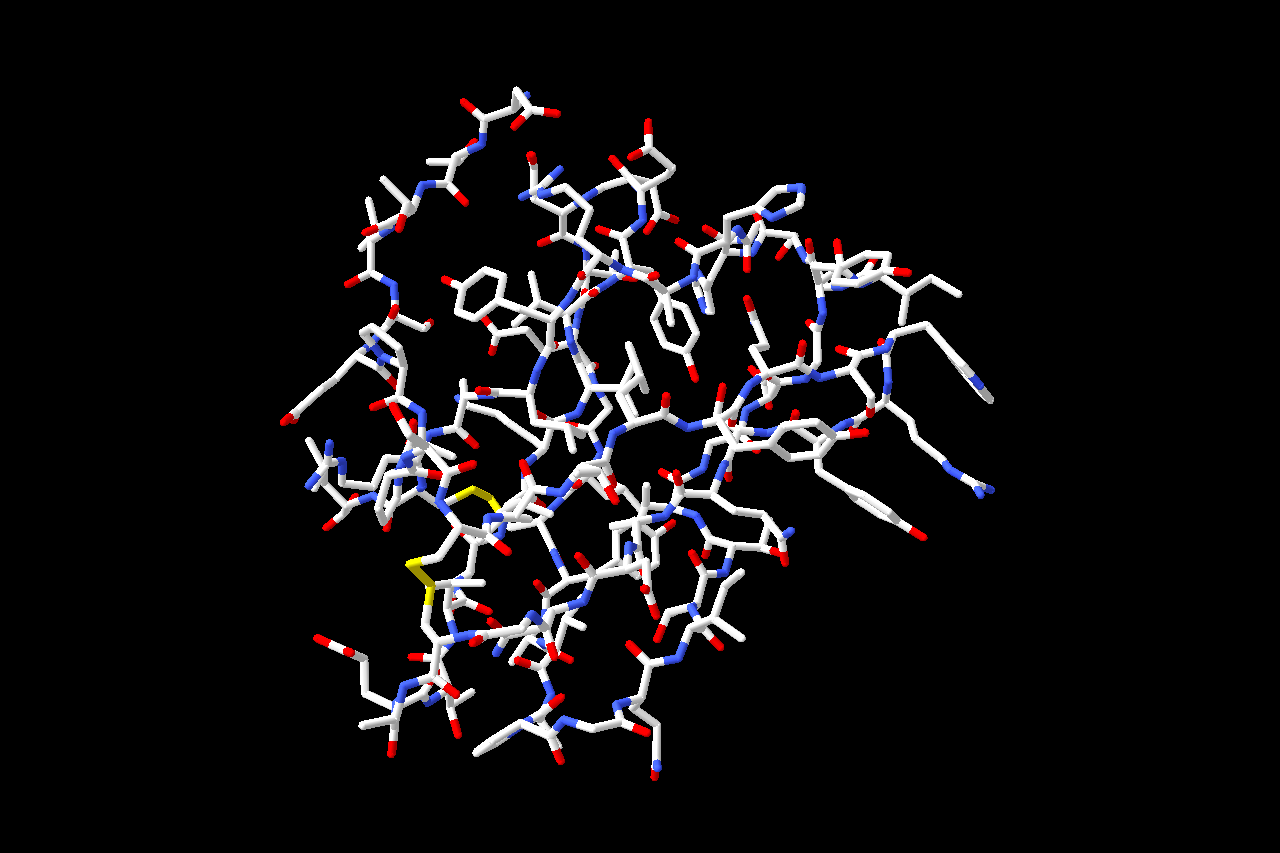}
              &  \includegraphics[width=0.32\textwidth]{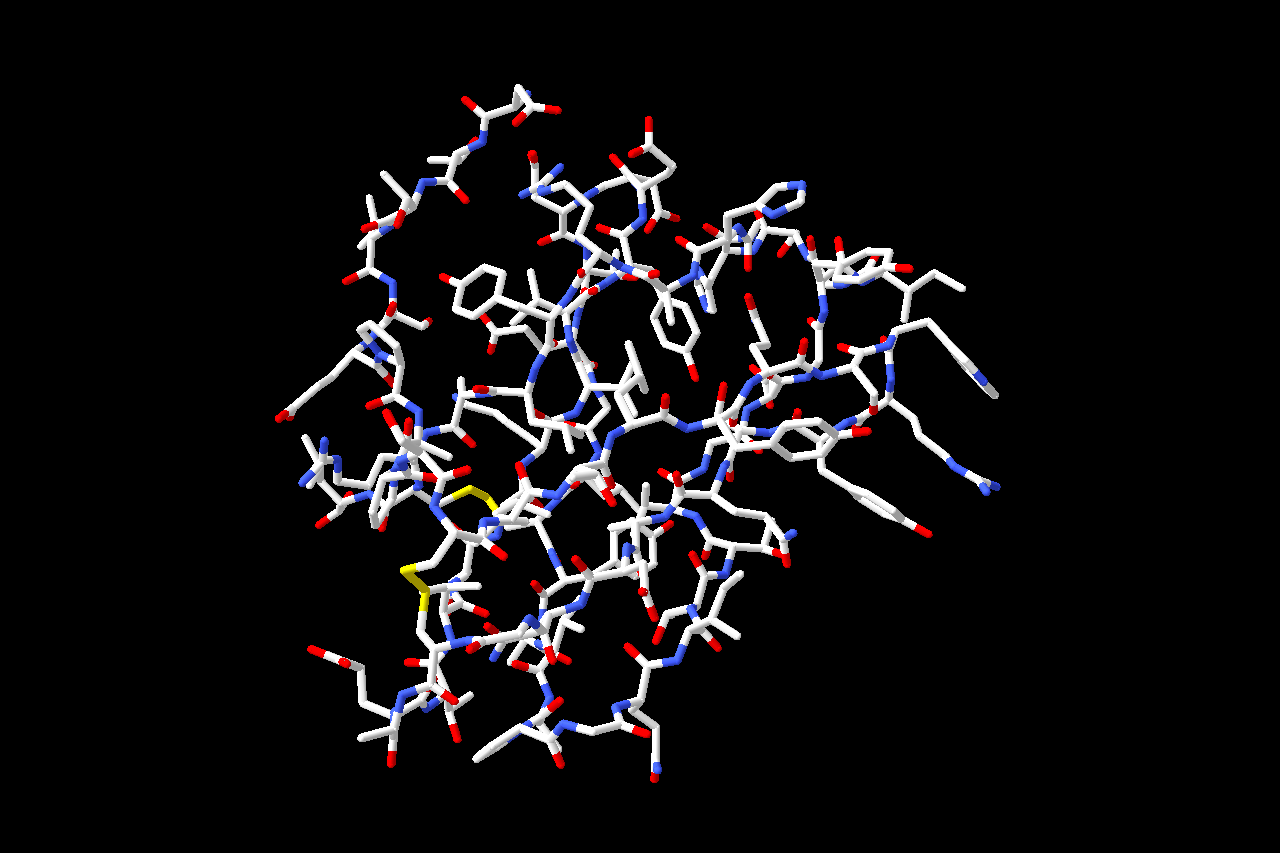} \\ \hline
    \rotatebox{90}{\hspace{32pt}1LFB}
              & \includegraphics[width=0.32\textwidth]{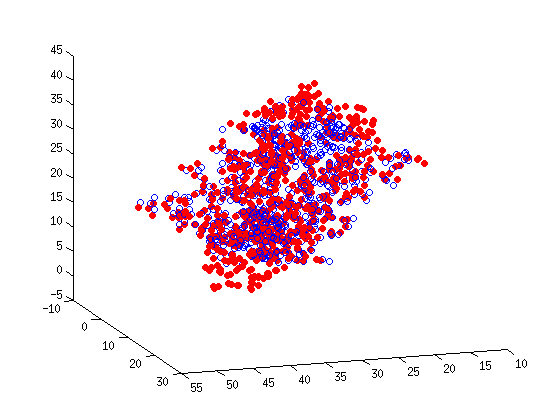}
              & \includegraphics[width=0.32\textwidth]{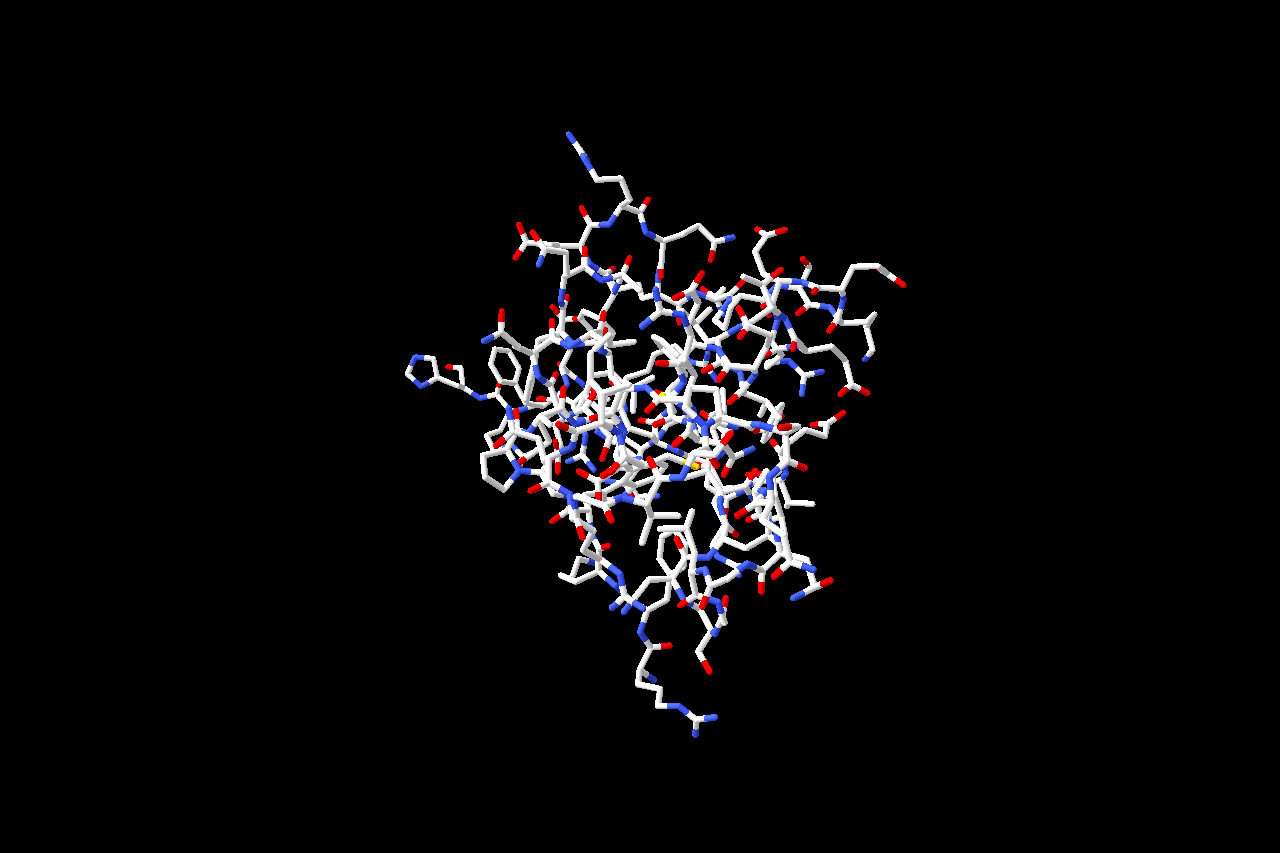}
              &  \includegraphics[width=0.32\textwidth]{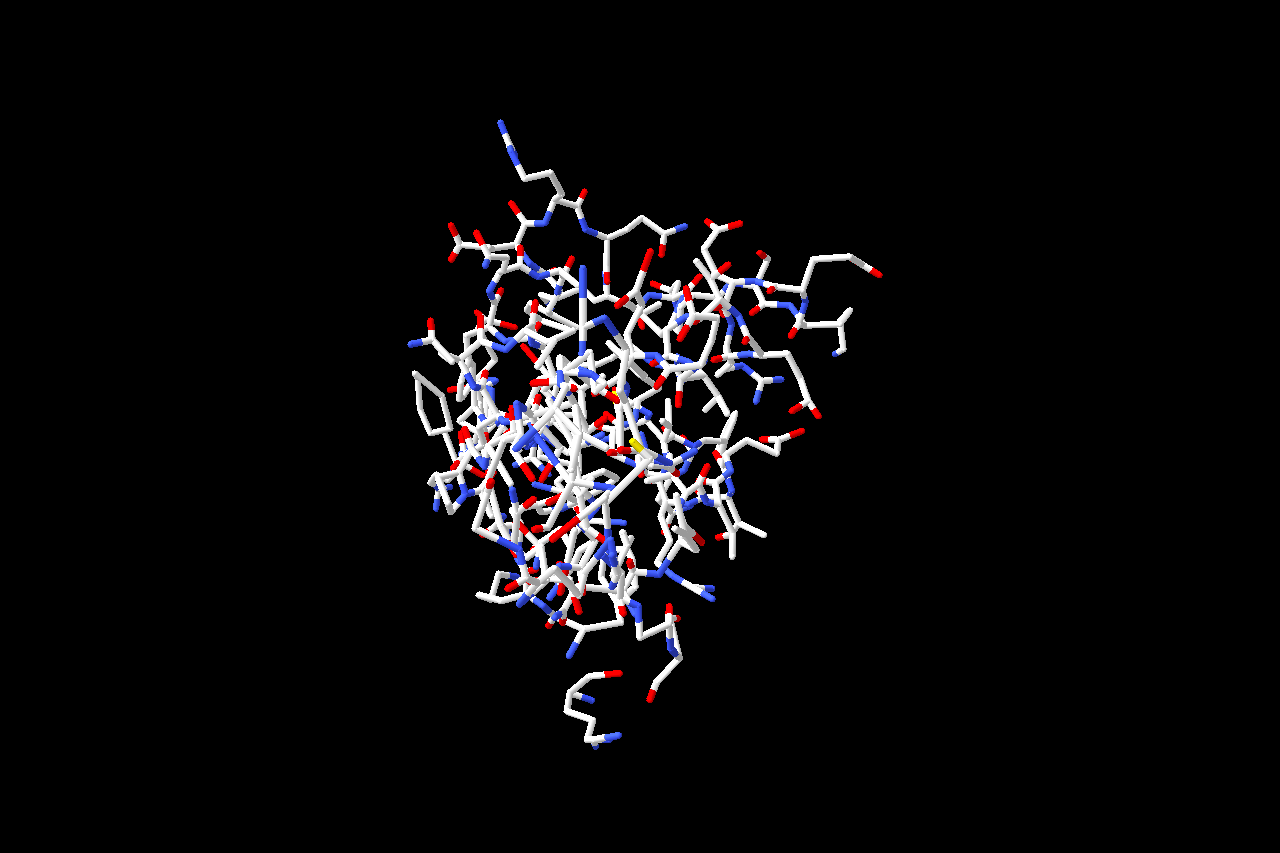} \\ \hline
    \rotatebox{90}{\hspace{32pt}1PHT}
              & \includegraphics[width=0.32\textwidth]{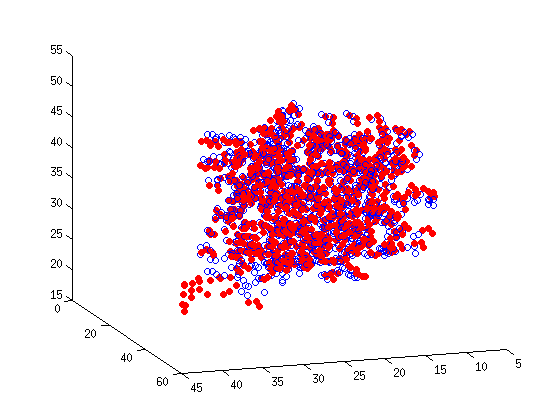}
              & \includegraphics[width=0.32\textwidth]{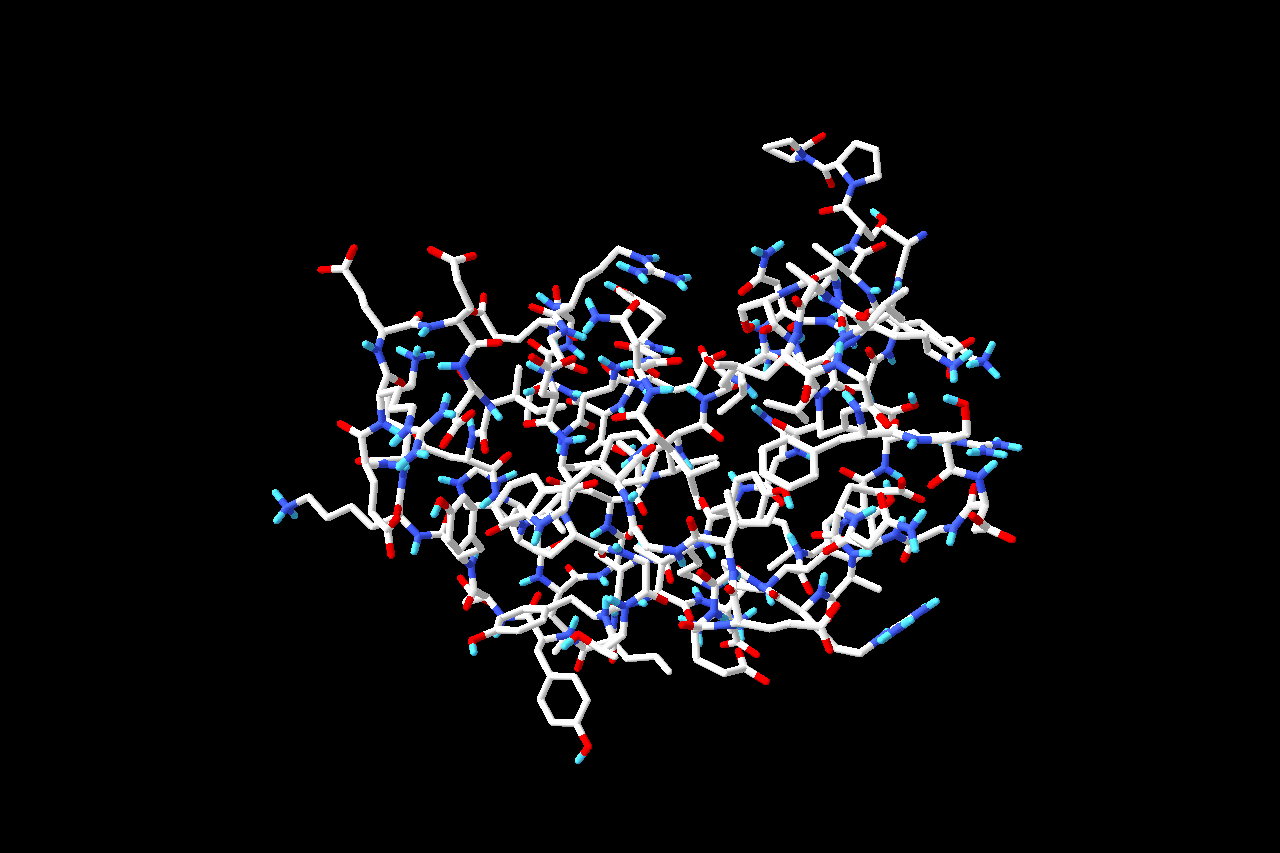}
              &  \includegraphics[width=0.32\textwidth]{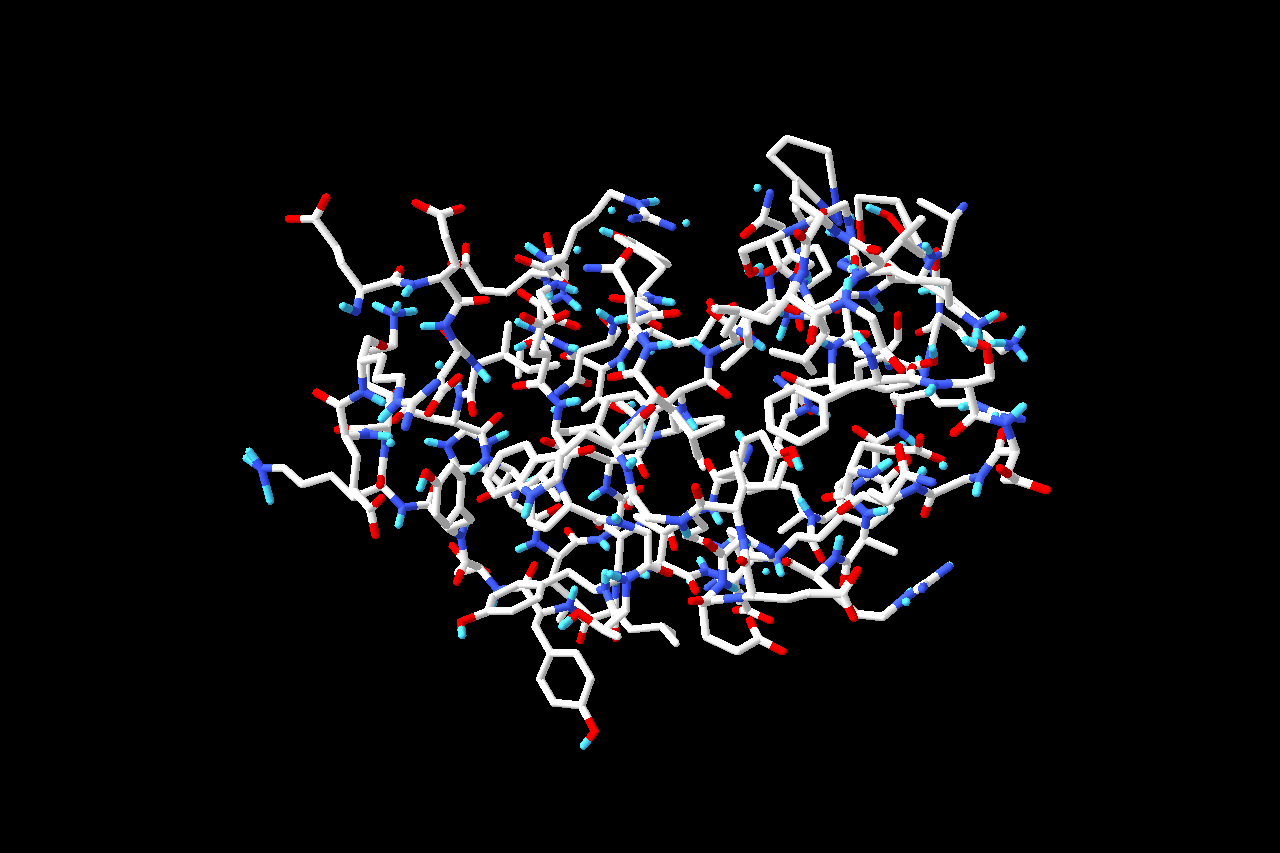} \\ \hline
    \rotatebox{90}{\hspace{32pt}1POA}
              & \includegraphics[width=0.32\textwidth]{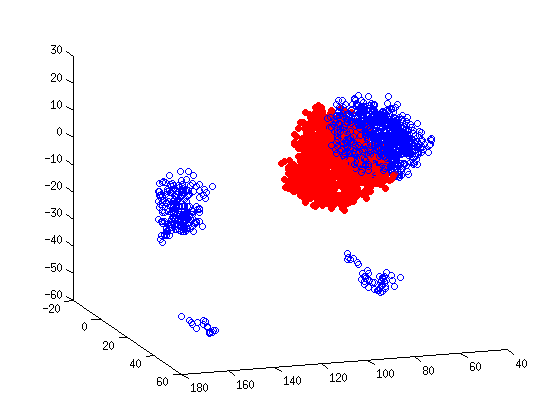}
              & \includegraphics[width=0.32\textwidth]{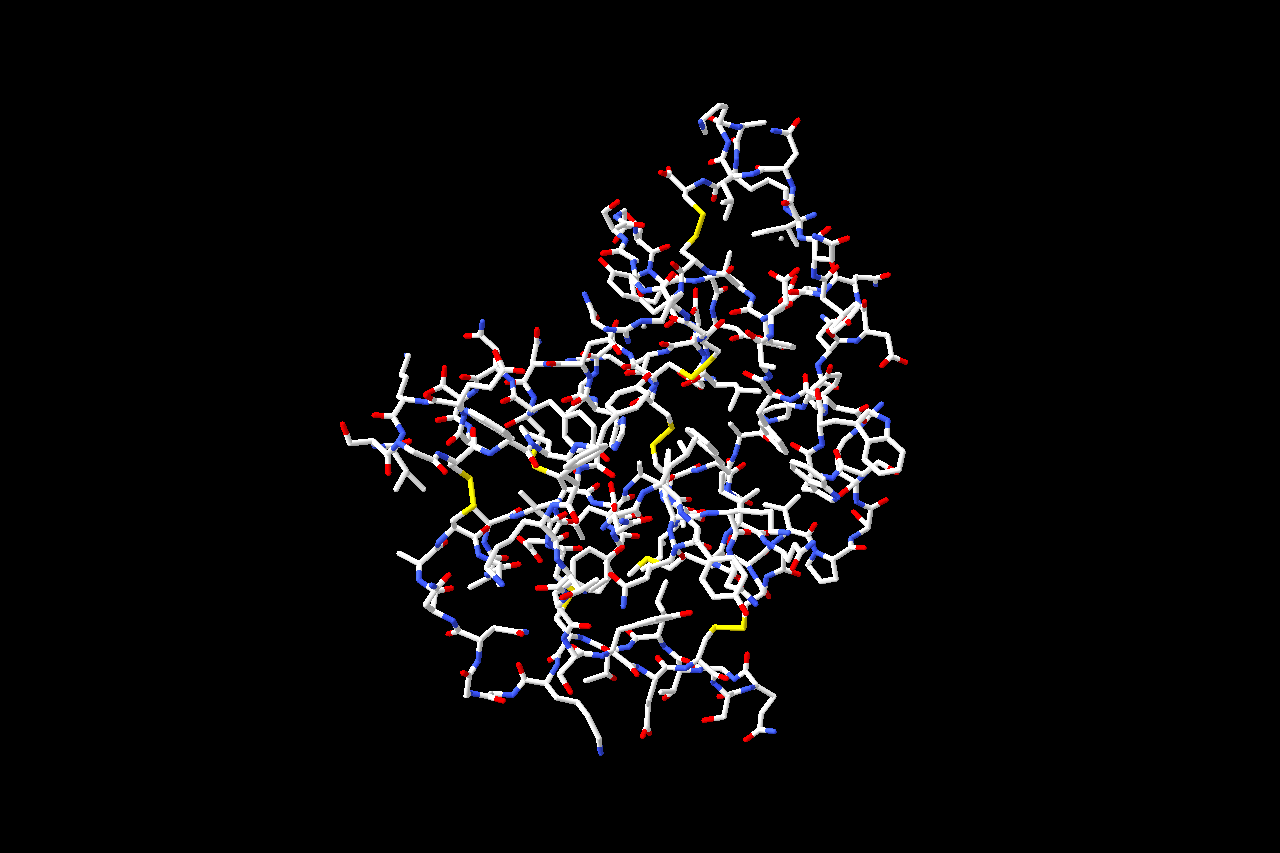}
              &  \includegraphics[width=0.32\textwidth]{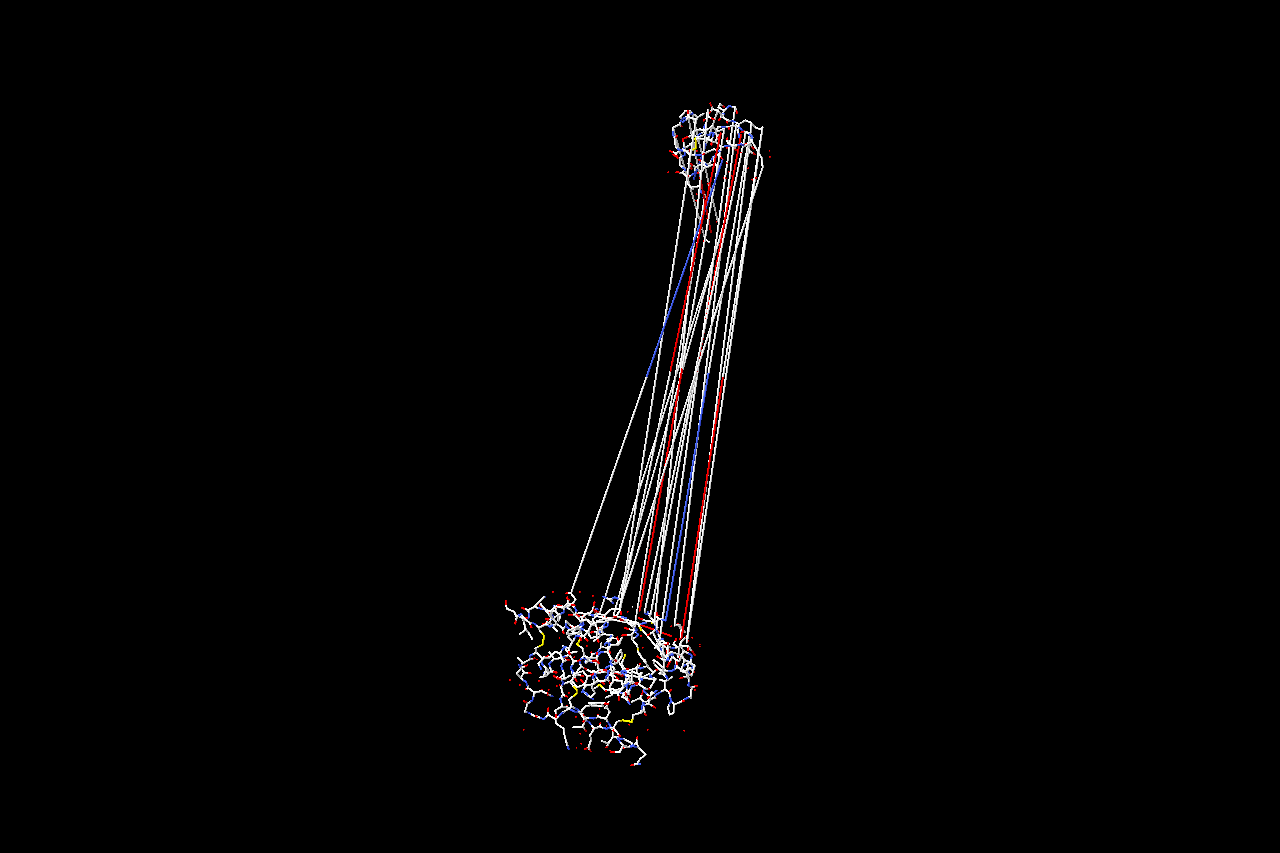} \\ \hline
    \rotatebox{90}{\hspace{32pt}1AX8}
              & \includegraphics[width=0.32\textwidth]{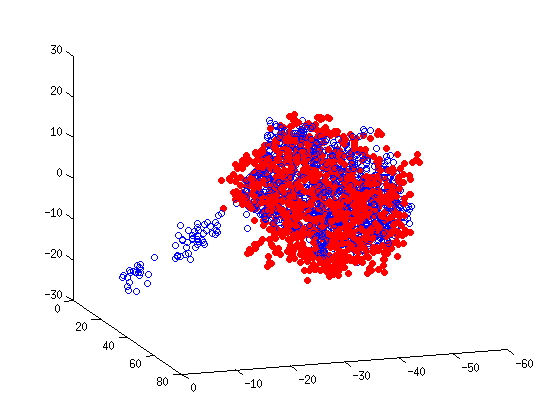}
              & \includegraphics[width=0.32\textwidth]{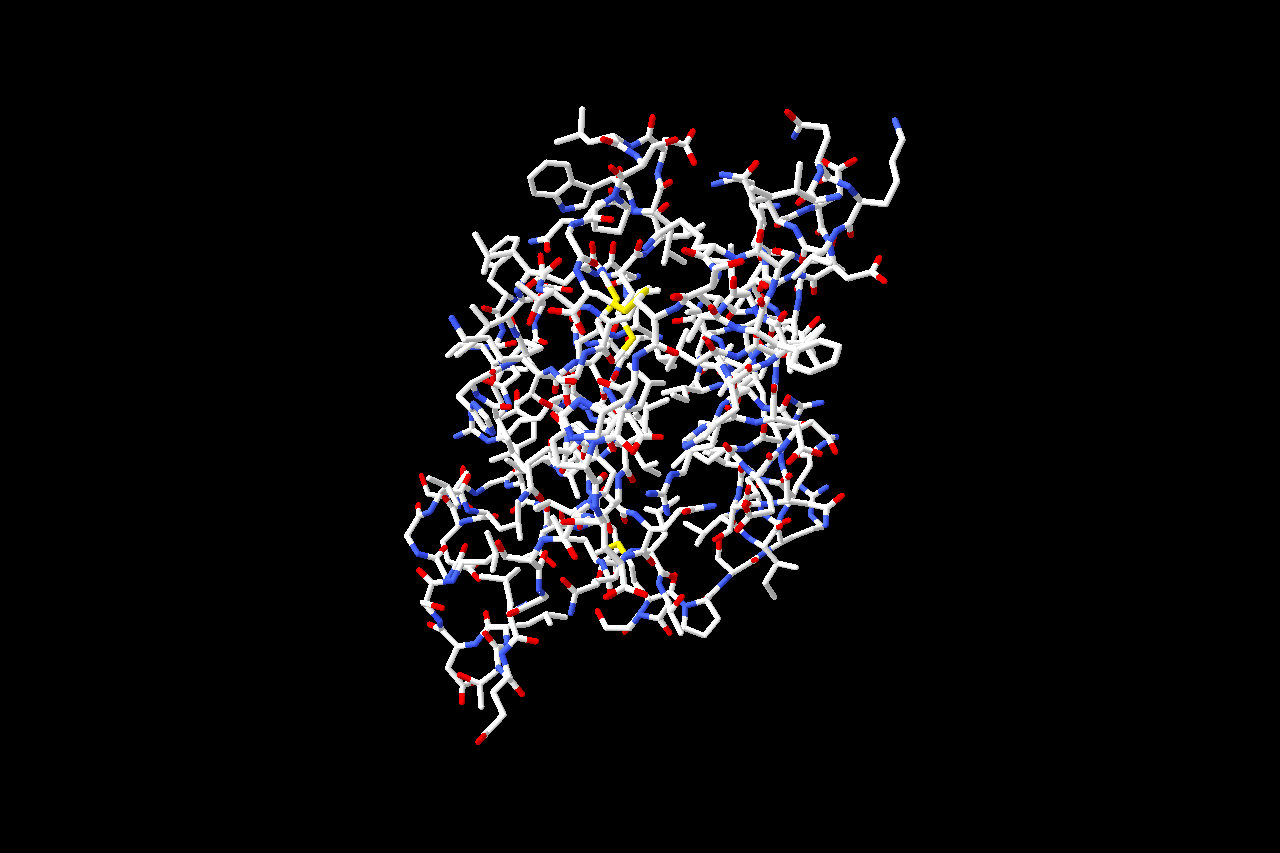}
              &  \includegraphics[width=0.32\textwidth]{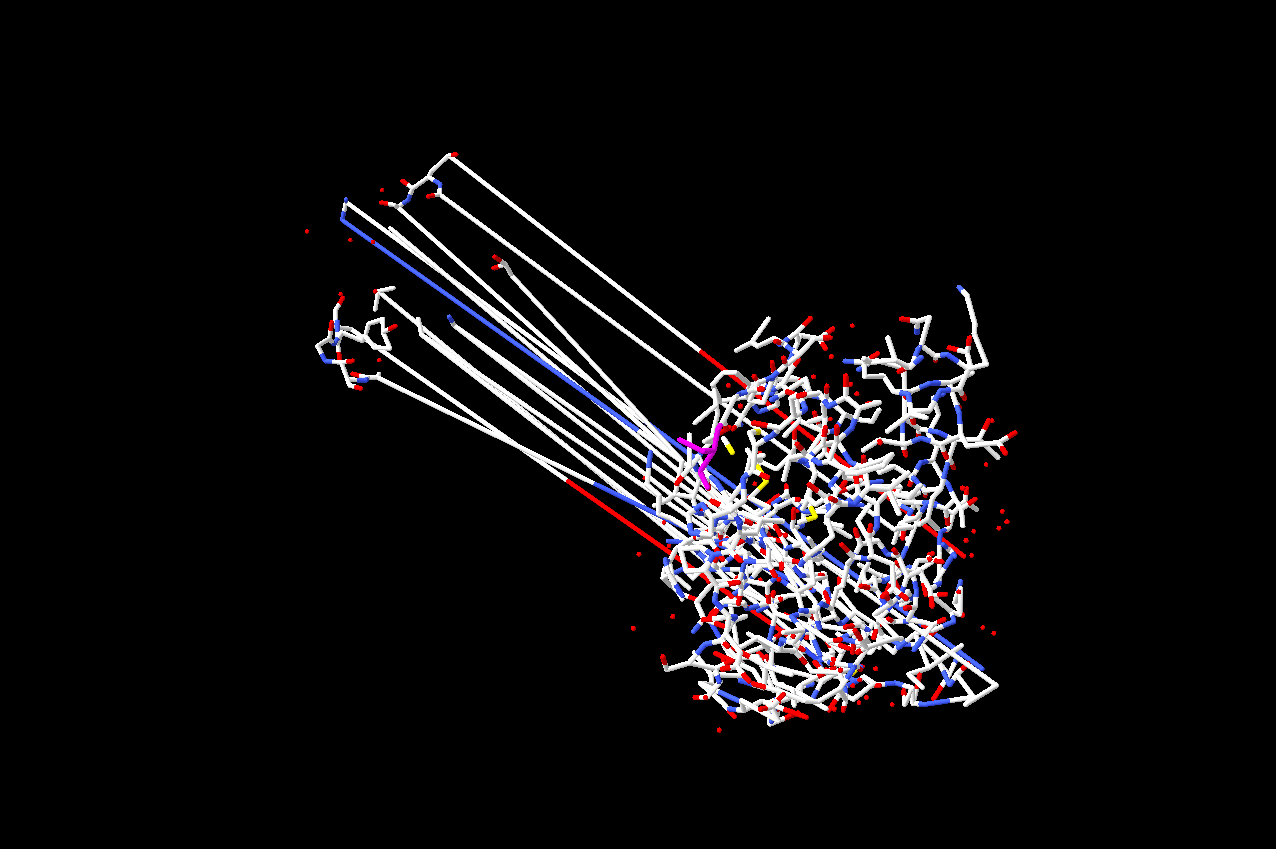} \\ \hline
    \end{tabular}
  \end{center}
  \end{adjustwidth}
\end{figure}

\begin{remark}[An upper bound on distances] \label{remark:protein1}
 The constraint $C_1$ can be easily modified to incorporate additional distance information. For instance, upper and lower bounds could be placed on the distance between (not necessarily adjacent) carbons atoms on a carbon chain. Note that each carbon-carbon bond is approximately 1.5\AA\ in length. \qede\end{remark}

\begin{remark}[Two phase approach]\label{remark:protein2}
 In our implementation, the Douglas--Rachford method encountered difficulties applied to the reconstruction of the two larger proteins. It therefore would be reasonable to consider an approach were one partitions the atoms into sets and applies the Douglas--Rachford to these sub-problems. The reconstructed distances obtained from these sub-problems can then be used as the initial estimates for distances in the original master problem (which considers all the atoms).
 An iterative version is outline in Algorithm~\ref{algor:2step}.

\medskip

\begin{algorithm}[H]
 \caption{A two phase algorithms for protein reconstruction.}\label{algor:2step}
 \SetKwInOut{Input}{input}
 \SetKwInOut{Output}{output}
 \Input{$D\in\R^{n\times n}$ (the partial Euclidean distances matrix)}
  Choose random $X\in[-1,1]^{n\times n}$\;
  $\Gamma:=\{1,2,\dots,n\}$ (each index represents an atom)\;
 \While{continue}{
   \eIf(\tcp*[f]{generate and solve sub-problems (phase 1)}){doPhase1}{
      Choose a partition of $\Gamma$ into the sets $\Gamma_1,\Gamma_2,\dots,\Gamma_m$\;
      \For{$k=1,2,\dots,m$}{
        Apply Algorithm~\ref{algor:DRprotein} to atoms indexed by $\Gamma_k$ to obtain $X_k$ (i.e. the distance matrix for the atoms indexed by $\Gamma_k$).\;
        Update $X$ with the reconstructed distances in $X_k$\;
      }
   }(\tcp*[f]{solve master problem {(phase 2)}})
   {
     Apply Algorithm~\ref{algor:DRprotein} to all atoms (i.e. index by $\Gamma$) to obtain $X$\;
   }
 }
 \Output{$X$ (the reconstructed distance matrix)}
 \end{algorithm}
\medskip
We continue to work on such problem-specific refinements of the Douglas-Rachford method: in most of our example problems a natural splitting is less accessible.
\qede\end{remark}

It would also  be interesting to  apply the methods of this section to \emph{sensor network localization} problems requiring the reconstruction of an incomplete distance matrix. See, for example, \cite{D06,KW10,GTSC13}.

\clearpage %please remove this line -- add to avoid float positing from the previous section

\subsection{Hadamard matrices}\label{sec:Hadamard}

Recall that a matrix $H=(H_{ij})\in\{-1,1\}^{n\times n}$ is said to be a \emph{Hadamard matrix} of \emph{order} $n$ if
 \begin{equation}\label{eq:hadamardC1}
  H^TH=nI.
 \end{equation}
We note that there are many equivalent characterizations. For instance, (\ref{eq:hadamardC1}) is equivalent to asserting that $H$ has maximal determinant (i.e. $|\det H|=n^{n/2}$)  \cite[Chapter~2]{H06}. A classical result of Hadamard asserts that Hadamard matrices exist only if $n=1,2$ or a multiple of $4$. For orders $1$ and $2$, such matrices are easy to find. For multiples of $4$, the \emph{Hadamard conjecture} asks the converse: \emph{If $n$ is a multiple of $4$, does there exists a Hadamard matrix of order $n$?} Background on Hadamard matrices can be found in \cite{H06}.   Thus, an important completion problem starts with structure restrictions, but with no fixed entries.

Consider the now the problem of finding a Hadamard matrix of a given order. We define the constraints:
 \begin{align}
  C_1 &:= \{X\in\R^{n\times n}|X_{ij}=\pm 1 \text{ for }i,j=1,\dots,n\}, \label{eq:HadamardC1} \\
  C_2 &:= \{X\in\R^{n\times n}|X^TX=nI\}.
 \end{align}
Then $X$ is a Hadamard matrix if and only if $X\in C_1\cap C_2$.

The first constraint, $C_1$, is clearly non-convex. However, its projection is simple and is given pointwise by
\begin{equation}
 P_{C_1}(X)_{ij} = \left\{\begin{array}{cl}
                    -1     & \text{if }X_{ij}<-1, \\
                    \pm 1 & \text{if }X_{ij}=0,\\
                    1      & \text{otherwise.}
                   \end{array}\right.
\end{equation}

The second constraint, $C_2$, is also non-convex. To see this, consider the mid-point of the two matrices
 $$ \begin{pmatrix} \sqrt{2} & 0 \\ 0 & \sqrt{2} \\ \end{pmatrix},
    \begin{pmatrix} 0 & \sqrt{2} \\ \sqrt{2} & 0 \\ \end{pmatrix} \in C_2.$$
Nevertheless, a projection can be found by solving the equivalent problem of finding a nearest orthogonal matrix --- a special case of the \emph{Procrustes problem} described above.

\begin{proposition}\label{prop:HadamardC2}
 Let $X=USV^T$ be a singular value decomposition. Then
 $$\sqrt{n}UV^T\in P_{C_2}(X).$$
\end{proposition}
\begin{proof}
 Let $Y=X/\sqrt{n}$. Then
  $$\min_{X\in\R^{n\times n}\atop A^TA=nI}\|X-A\|_F=\sqrt{n}\left(\min_{Y\in\R^{n\times n}\atop B^TB=I}\|Y-B\|_F\right).$$
 Any solution to the latter is the nearest orthogonal matrix to $Y$. One such matrix can be obtained by replacing all singular values of $Y$ by `one' (see, for example, \cite{S66}). Since
  $$Y=U\hat SV^T\text{ where }\hat S=S/\sqrt{n},$$
is a singular value decomposition, it follows that $UV$ is the nearest orthogonal matrix to $Y$. The result now follows.
\end{proof}

\begin{remark}
 Any $A\in P_{C_2}(X)$ is such that $\tr(A^TX)=\max_{B\in C_2} \tr(B^TX)$. \qede
\end{remark}

\begin{remark}
 Consider instead the matrix completion problem of finding a Hadamard matrix with some known entries.  This can be cast within the above framework by appropriate modification of $C_1$. The projection onto $C_1$  only differs by leaving the known entries unchanged. \qede
\end{remark}

We next give a second useful formulation for the problem of finding a Hadamard matrix of a given order. We take $C_1$ as in (\ref{eq:hadamardC1}) and define
 $$C_3:=\{X\in\R^{n\times n}|X^TX=\|X\|_F\, I\}.$$
If $X\in C_1$ then $\|X\|_F=n$, hence $C_1\cap C_2=C_1\cap C_3$. It follows that $X$ is a Hadamard matrix if and only if $X\in C_1\cap C_3$. A projection onto $C_3$ is given similarly  $P_{C_2}$.

\begin{proposition}\label{prop:HadamardC3}
Let $X=USV^T$ be a singular value decomposition. Then
 $$\sqrt{\|X\|_F}\,UV^T\in P_{C_3}(X).$$
\end{proposition}
\begin{proof}
 This is a straightforward modification of Proposition~\ref{prop:HadamardC2}.
\end{proof}

\begin{remark}[Complex Hadamard matrices]\label{rem:h2} It is also possible to consider \emph{complex Hadamard matrices}. In this case,
 $$C_1:=\{X\in\mathbb{C}^{n\times n}|~|X_{ij}|=1\}.$$
The projection onto $C_1$ is straightforward, and is given by
  $$P_{C_1}(X)_{ij}=\left\{\begin{array}{cl}
                     X_{ij}/|X_{ij}| & \text{if }X_{ij}\neq 0, \\
                     C_1             & \text{otherwise.} \\
                    \end{array}\right.$$
Note that the real solutions to $|X_{ij}|=1$ are $\pm 1$.\qede
\end{remark}

\begin{example}[Experiments with Hadamard matrices]\label{ex:had}
Let $H_1$ and $H_2$ be Hadamard matrices. We say $H_1$ are $H_2$ are \emph{distinct} if $H_1\neq H_2$. We say $H_1$ and $H_2$ are \emph{equivalent} if $H_2$ can be obtained from $H_1$ by performing a sequence of row/column permutations, and/or multiplying row/columns by $-1$. The number of distinct (resp. inequivalent) Hadamard matrices of order $4n$ is given in OEIS sequence \href{http://oeis.org/A206712}{A206712} :768, 4954521600, 20251509535014912000,... (resp. \href{http://oeis.org/A007299}{A00729}: 1, 1, 1, 1, 5, 3, 60, 487, 13710027, ...). With increasing order, the number of Hadamard matrices is a faster than exponentially decreasing proportion of the total number of $\{+1,-1\}$-matrices (of which there are $2^{n^2}$ for order $n$). This is reflected in the observed rapid increase in difficulty of finding Hadamard matrices using the Douglas--Rachford scheme,  as order increases.

We applied the Douglas--Rachford algorithm to 1000 random replications, for each of the above formulation. Our computational experience is summarized in Table~\ref{table2} and Figure~\ref{fig:Hadamard}. To determine if two Hadamard matrices are equivalent, we use a \emph{Sage} implementation of the graph isomorphism approach outlined in \cite{K79}.

\begin{table}[htb]
 \begin{adjustwidth}{-0.5in}{-0.5in} %please remove adjustwidth --- centering figure for now...
 \begin{center}
   \caption{Number of Hadamard matrices found from 1000 instances.}\label{table2}
   {\footnotesize
   \begin{tabular}{|c|cccc|cccc|} \hline
    \multirow{2}{*}{Order} & \multicolumn{4}{c|}{Prop.~\ref{prop:HadamardC2} Formulation}  & \multicolumn{4}{|c|}{Prop.~\ref{prop:HadamardC3} Formulation} \\ \cline{2-9}
                           & Ave Time (s) & Solved & Distinct & Inequivalent & Ave Time (s) & Solved & Distinct & Inequivalent \\ \hline
    2                      & 1.1371       & 534    & 8        & 1            & 1.1970       & 505    & 8        & 1 \\
    4                      & 1.0791       & 627    & 422      & 1            & 0.2647       & 921    & 541      & 1 \\
    8                      & 0.7368       & 996    & 996      & 1            & 0.0117       & 1000   & 1000     & 1 \\
    12                     & 7.1298       & 0      & 0        & 0            & 0.8337       & 1000   & 1000     & 1 \\
    16                     & 9.4228       & 0      & 0        & 0            & 11.7096      & 16     & 16       & 4 \\
    20                     & 20.6674      & 0      & 0        & 0            & 22.6034      & 0      & 0        & 0 \\ \hline
   \end{tabular}
   }
 \end{center}
 \end{adjustwidth}
\end{table}

We make some brief comments our the results.
\begin{itemize}\item

The formulation based on Proposition~\ref{prop:HadamardC3} was found to be faster and more successful than the formulation based on Proposition~\ref{prop:HadamardC2}, especially for orders $8$ and $12$ where it was successful in every replication. For order less than or equal to 12, the Douglas--Rachford schema was able to find the unique inequivalent Hadamard matrix under either formulation (except for $n=12$, Prop.~\ref{prop:HadamardC2} formulation). Moreover, the Proposition~\ref{prop:HadamardC3} formulation was able to find four of the five inequivalent Hadamard matrices of order 16.\footnote{All five can be found at \url{http://www.uow.edu.au/~jennie/hadamard.html}.}
\item
From Figure~\ref{fig:Hadamard} we observe that if a Hadamard matrix was found, it was usually found within the first few thousand iterations. The frequency histogram for order 16, shown in Figure~\ref{fig:Hadamard}(f), varied significantly from the corresponding histograms for lower orders.
\end{itemize}

\begin{figure}
  \begin{center}
	  \begin{subfigure}{0.49\textwidth}
	  \includegraphics[width=\textwidth]{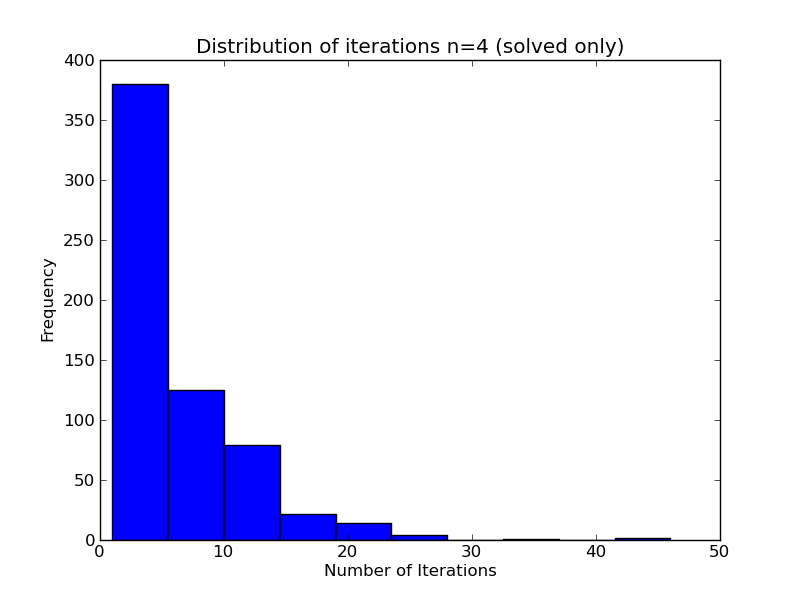}\caption{$n=4$, Prop.~\ref{prop:HadamardC2} formulation.}
	  \end{subfigure}
	  \begin{subfigure}{0.49\textwidth}
	  \includegraphics[width=\textwidth]{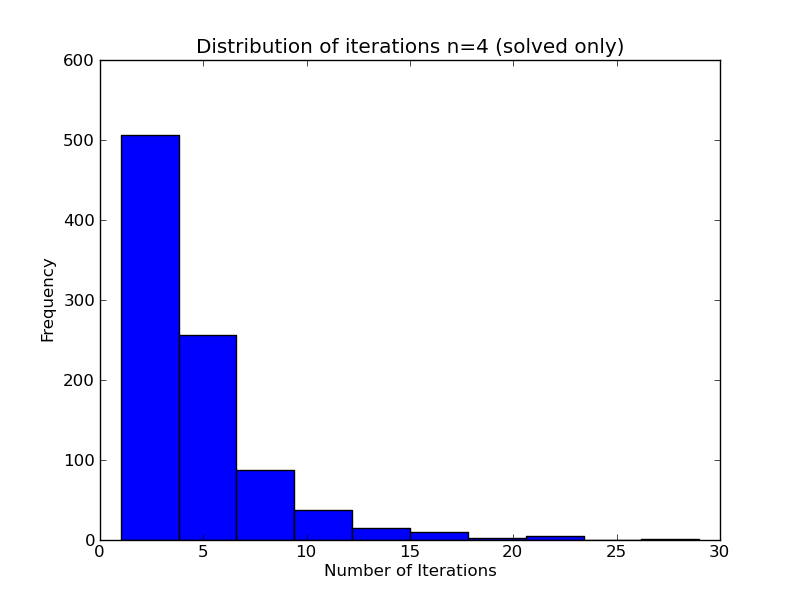}\caption{$n=4$, Prop.~\ref{prop:HadamardC3} formulation.}
	  \end{subfigure}
	  \begin{subfigure}{0.49\textwidth}
	  \includegraphics[width=\textwidth]{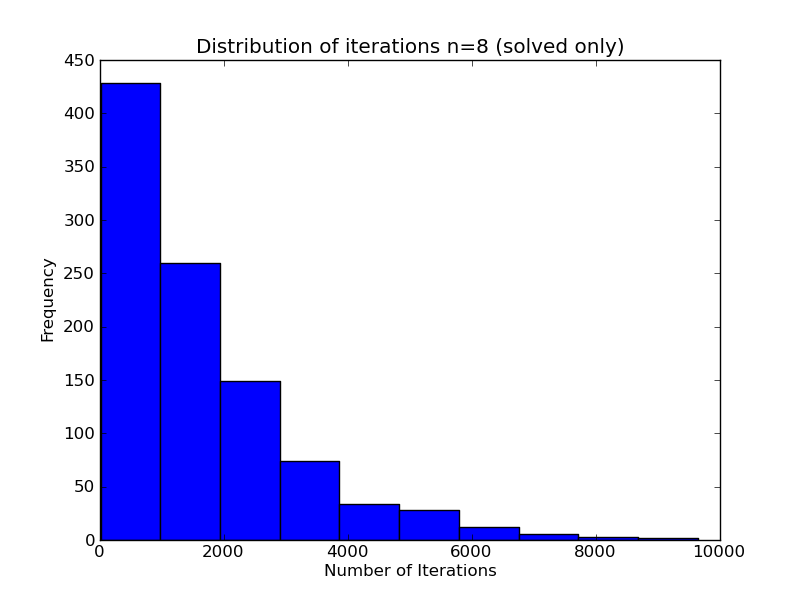}\caption{$n=8$, Prop.~\ref{prop:HadamardC2} formulation.}
	  \end{subfigure}
	  \begin{subfigure}{0.49\textwidth}
	  \includegraphics[width=\textwidth]{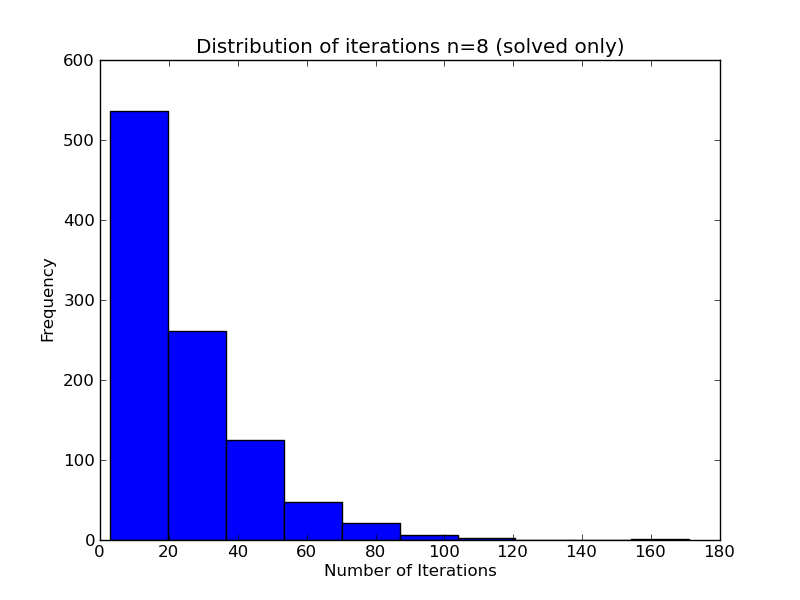}\caption{$n=8$, Prop.~\ref{prop:HadamardC3} formulation.}
	  \end{subfigure}
	  \begin{subfigure}{0.49\textwidth}
	  \includegraphics[width=\textwidth]{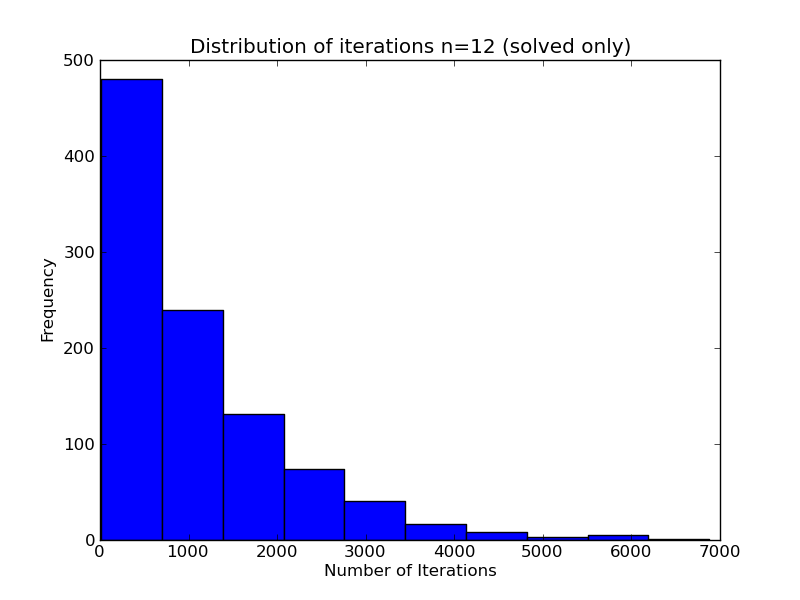}\caption{$n=12$, Prop.~\ref{prop:HadamardC3} formulation.}
	  \end{subfigure}
	  \begin{subfigure}{0.49\textwidth}
	  \includegraphics[width=\textwidth]{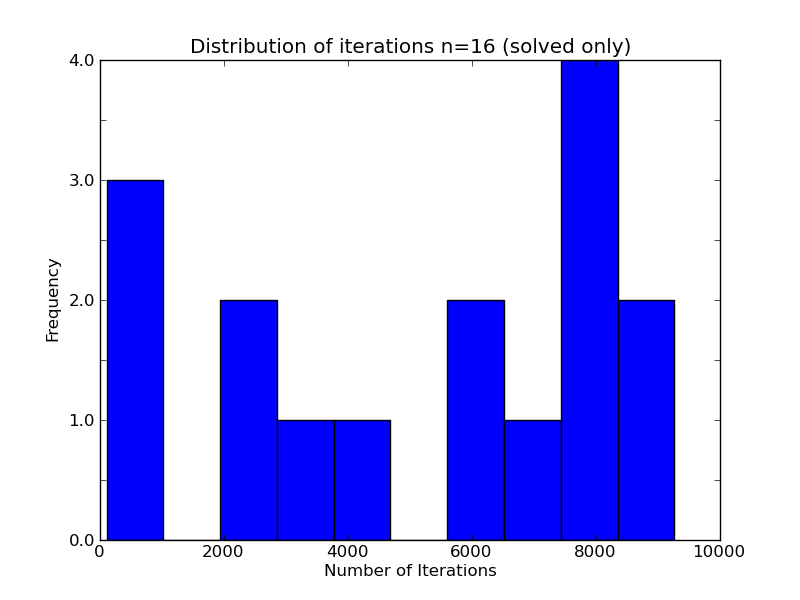}\caption{$n=16$, Prop.~\ref{prop:HadamardC3} formulation.}
	  \end{subfigure}
  \end{center}
  \caption{Frequency histograms showing the number of iterations required to find a Hadamard matrix, for different order and formulations (solved instances only).} \label{fig:Hadamard}
\end{figure}

For orders 20 and above, it is possible that another formulation might be more fruitful, but almost certainly better and more problem-specific heuristics will again be needed. \qede

\end{example}

\begin{remark}
 Since $C_2$ is non-convex, when computing its projection we are forced to make a selection from the set of nearest points. In our experiments we have always chosen the nearest point in the same way. It maybe possible to benefit from making the selection according to some other criterion.  \qede
\end{remark}

We now turn our attention to some special classes of Hadamard matrices.

\subsubsection{Skew-Hadamard matrices}
Recall that a matrix $A=(a_{ij})\in\R^{n\times n}$ is \emph{skew-symmetric} if $A^T=-A$. A \emph{skew-Hadamard matrix} is a Hadamard matrix, $H$, such that $(I-H)$ is skew-symmetric. That is,
 $$H+H^T=2I.$$
Skew-Hadamard matrices are of interest, for example, in the construction of \emph{combinatorial designs}. (For a survey see \cite{KS08}.) The number of inequivalent skew-Hadamard matrices of order $4n$ is given in OEIS sequence \href{http://oeis.org/A001119}{A001119}: 1, 1, 2, 2, 16, 54, \dots (for $n=2,3,\dots$).

In addition to the constraints $C_1$ and $C_2$ from the previous section, we define the affine constraint
 $$C_3 := \{X\in\R^{n\times n}|X+X^T=2I\}.$$

A projection onto $C_1\cap C_3$ is given by
 $$P_{C_1\cap C_3}(X)_{ij}=\left\{\begin{array}{cl}
                       1  & \text{if }i\neq j\text{ and }X_{ij}\geq X_{ji}, \\
                       -1 & \text{if }i\neq j\text{ and }X_{ij}<X_{ji}, \\
                       1 & \text{otherwise.}
                      \end{array}\right.$$

Then $X$ is a skew-Hadamard matrix if and only if $X\in(C_1\cap C_3)\cap C_2$.

Table~\ref{tab:skewHad} shows the results of the same experiment as Section~\ref{sec:Hadamard}, but with the skew constraint incorporated.

\begin{table}[htb]
 \begin{adjustwidth}{-1in}{-1in} %please remove adjustwidth --- centering figure for now...
 \begin{center}
   \caption{Number of skew-Hadamard matrices found from 1000 instances.}\label{tab:skewHad}
   {\footnotesize
   \begin{tabular}{|c|cccc|cccc|} \hline
    \multirow{2}{*}{Order} & \multicolumn{4}{c|}{Prop.~\ref{prop:HadamardC2} Formulation}  & \multicolumn{4}{|c|}{Prop.~\ref{prop:HadamardC3} Formulation} \\ \cline{2-9}
                           & Ave Time (s) & Solved & Distinct & Inequivalent & Ave Time (s) & Solved & Distinct & Inequivalent \\ \hline
    2                      & 0.0003       & 1000   & 2        & 1            & 0.0004       & 1000   & 2        & 1 \\
    4                      & 1.1095       & 719    & 16       & 1            & 1.6381       & 634    & 16       & 1 \\
    8                      & 0.7039       & 902    & 889      & 1            & 0.0991       & 986    & 968      & 1 \\
    12                     & 14.1835      & 43     & 43       & 1            & 0.0497       & 999    & 999      & 1 \\
    16                     & 19.3462      & 0      & 0        & 0            & 0.2298       & 1000   & 1000     & 2 \\
    20                     & 29.0383      & 0      & 0        & 0            & 20.0296      & 495    & 495      & 2 \\ \hline
   \end{tabular}
   }
 \end{center}
 \end{adjustwidth}
\end{table}

\begin{remark}
Comparing the results of Table~\ref{tab:skewHad} with those of Table~\ref{table2}, it is notable that by placing additional constraints on the problem, both methods now succeed at higher orders, method two is faster than before, and  we can successfully find all inequivalent skew matrices of order 20 or less.

In contrast, the three-set feasibility problem $C_1\cap C_2\cap C_3$ was unsuccessful, except for order 2. This is despite the projection onto the affine set $C_3$ having the simple formula
\begin{equation}
 P_{C_3}(X) = I+\frac{X-X^T}{2}.
\end{equation}
Many mysteries remain!
\qede

\end{remark}

\subsubsection{Circulant Hadamard matrices}

Recall that a matrix $A=(a_{ij})\in\R^{n\times n}$ is \emph{circulant} if it can be expressed as
 $$A=\begin{pmatrix}
      \lambda_1 & \lambda_2 & \dots  & \lambda_n \\
      \lambda_n & \lambda_1 & \dots  & \lambda_{n-1} \\
      \vdots    & \vdots    & \ddots & \vdots \\
      \lambda_2 & \lambda_3 & \dots  & \lambda_1 \\
     \end{pmatrix}$$
for some vector $\lambda\in\R^n$.

The set of circulant matrices form a subspace of $\R^{n\times n}$. The set $\{P^k:i=1,2,\dots,n\}$, where $P$ is the cyclic permutation matrix
   $$P:=\begin{pmatrix}
      0 & 0 & \dots & 0 & 1 \\
      1 & 0 & \dots & 0 & 0 \\
      \vdots & \vdots & \ddots &\vdots & \vdots \\
      0 & 0 & \dots & 1 & 0 \\
     \end{pmatrix},$$
forms a basis. Consequently, any circulant matrix, $A$, can be expressed as the linear combination of the form
  $$A=\sum_{k=1}^n\lambda_k P^k.$$

\begin{remark}
 Right (resp. left) multiplication by $P$ results in a cyclic permutation of rows (resp. columns). Hence $P^2,P^3,\dots,P^n$ represent all cyclic permutations of the rows (resp. columns) of $P$. In particular, $P^n$ is the identity matrix.
 \qede \end{remark}

\begin{proposition}[{\cite[Ex.~6.7]{H11}}]
 For $X\in\R^{n\times n}$, the nearest circulant matrix is given by
  $$\sum_{k=1}^n\lambda_kP^k\text{ where } \lambda_k = \frac{1}{n}\sum_{i,j} P^k_{ij}X_{ij}.$$
\end{proposition}

A \emph{circulant Hadamard matrix} is a Hadamard matrix which is also circulant.

The circulant Hadamard conjecture asserts: \emph{No circulant Hadamard matrix of order larger than $4$ exists.} For recent progress on the conjecture, see \cite{LS12}. Consistent with this conjecture, our Douglas--Rachford implementation can successfully find circulant matrices of order 4, but fails for higher orders.

\section{Conclusion}\label{sec:conc}

We have provided general guidelines for successful application of the  Douglas Rachford method to (real) matrix completion problems: both convex and non-convex. The message of the previous two sections is the following.   When presented with a new (potentially non-convex) feasibility problem it is well worth seeing if Douglas--Rachford can deal with it--- as it is both conceptually very simple and is usually relatively easy to implement. If it works one may then think about refinements if performance is less than desired.

Moreover, this approach allows for the intuition developed for continuous optimization in Euclidean space to be usefully repurposed. This also lets one profitably consider non-expansive fixed point methods in the class of so-called \emph{CAT(0) metric spaces} --- a far ranging concept introduced twenty years ago in algebraic topology but now finding applications to optimization and fixed point algorithms.
The convergence of various projection type algorithms to feasible points is under investigation by Searston and Sims among others in such spaces \cite{cat} --- thereby broadening the constraint structures to which projection-type algorithms apply to include metrically rather than only algebraically convex sets.

Future computational experiments could include:

 \begin{itemize} \item Implementing  the modifications to the protein reconstruction formulation outlined in Remarks~\ref{remark:protein1} and \ref{remark:protein2}.

  \item  Consideration of similar reconstruction problems arising in the context of ionic liquid chemistry, and as mentioned, sensor location problems.

\item Likewise, for the discovery of larger Hadamard matrices to be tractable by Douglas--Rachford methods, a more efficient implementation is needed and a more puissant model.
    \end{itemize}

\paragraph{Acknowledgments} We wish to thank Judy-anne Osborn and Richard Brent for useful discussions about Hadamard matrices, David Allingham for useful discussions on correlation matrices, Henry Wolkowicz for directing us to many useful matrix completion resources, and Brailey Sims for careful reading of the manuscript.

Many resources can be found at the paper's companion website:\\
\centerline{\url{http://carma.newcastle.edu.au/DRmethods}}

%\newpage
\footnotesize

 \bibliographystyle{plain}

\end{document}